\documentclass[11pt]{amsart}
\usepackage[margin=1in]{geometry}
\usepackage[foot]{amsaddr}

\usepackage[utf8]{inputenc} 
\usepackage[T1]{fontenc}    
\usepackage{url}            
\usepackage{booktabs}       
\usepackage{amsfonts}       
\usepackage{nicefrac}       
\usepackage{microtype}      
\usepackage{bbm}

\usepackage[colorlinks=true,linkcolor=blue!70!black,urlcolor=black,citecolor=green!60!black]{hyperref}

\usepackage{amsmath,amsthm,amssymb}
\usepackage{enumitem}
\usepackage{comment}
\usepackage{color}
\usepackage{mathtools}
\usepackage{fixmath}
\usepackage{algorithm} 
\usepackage{algpseudocode} 
\usepackage{subcaption}

\usepackage{tikz}
\usepackage{tikzscale}

\newtheorem{theorem}{Theorem}

\newtheorem{proposition}{Proposition}
\newtheorem{lemma}{Lemma}

\newtheorem{corollary}{Corollary}
\newtheorem{definition}{Definition}
\newtheorem{remark}{Remark}

\usepackage{booktabs}
\usepackage{enumitem}
\usepackage{siunitx}
\newcommand{\R}{\mathbb{R}}

\newcommand{\N}{\mathbb{N}}

\DeclareMathOperator*{\argmin}{arg\,min}

\DeclarePairedDelimiter{\abs}{\lvert}{\rvert}
\DeclarePairedDelimiter{\norm}{\lVert}{\rVert}
\DeclarePairedDelimiter{\pair}{\langle}{\rangle}
\DeclarePairedDelimiter{\inner}{(}{)}
\newcommand{\de}{\mathop{}\!\mathrm{d}}
\newcommand{\dd}{\de}		
\newcommand{\eps}{\varepsilon}
\newcommand\restr[2]{{
  \left.\kern-\nulldelimiterspace 
  #1 
  \vphantom{\big|} 
  \right|_{#2} 
  }}

\DeclareMathOperator*{\supp}{supp}
\DeclareMathOperator*{\atom}{atom}

\DeclareMathOperator*{\sign}{sign}
\DeclareMathOperator{\Prox}{Prox}

\newcommand{\NN}{\mathcal{N}}
\newcommand{\Sd}{\mathbb{S}^d}
\newcommand{\Hi}{L^2(D,\nu)}
\newcommand*{\Lap}{\upDelta}

\usepackage{pgfplots}
  \pgfplotsset{compat=newest}
  \usetikzlibrary{plotmarks}
  \usetikzlibrary{arrows.meta}
  \usepgfplotslibrary{patchplots}
  \usepackage{grffile}

\pgfplotsset{plot coordinates/math parser=false}
 \newlength\figureheight
 \newlength\figurewidth
 
\usetikzlibrary{external}
\pgfkeys{/pgf/images/include external/.code=\includegraphics{#1}}

\graphicspath{{./tikz/}{./images/}}

\newcommand*{\STATE}{\State}
\newcommand*{\WHILE}{\While}
\newcommand*{\ENDWHILE}{\EndWhile}

\usepackage[textsize=small,color=red!70]{todonotes}
\makeatletter
\renewcommand{\todo}[2][]{\tikzexternaldisable\@todo[#1]{#2}\tikzexternalenable}
\makeatother

\definecolor{dgreen}{RGB}{0, 100, 0}
\newcommand*{\fixed}[1]{#1}

\begin{document}

\title[Nonconvex regularization for sparse neural networks]{Nonconvex regularization for sparse neural networks}

\author{Konstantin Pieper}
\address{%
  Computer Science and Mathematics Division,
  Oak Ridge National Laboratory, One Bethel Valley Road, P.O. Box 2008, MS-6211, Oak Ridge, TN 37831\\
 }
\email{pieperk@ornl.gov}
\author{Armenak Petrosyan}
\address{
  School of Mathematics, Georgia Institute of Technology, 686 Cherry St NW, Atlanta, GA 30332\\
}
\email{petrosyan@math.gatech.edu}
\thanks{
The material in this manuscript is based on work supported by the Laboratory Directed
Research and Development Program at Oak
Ridge National Laboratory (ORNL), managed by UT-Battelle, LLC, under Contract
No.\ DE-AC05-00OR22725 and by the U.S.\ Department of Energy, Office of
Science, Early Career Research Program under award number ERKJ314.
The US government retains and the publisher, by accepting the article for publication, acknowledges that the US government retains a nonexclusive, paid-up, irrevocable, worldwide license to publish or reproduce the published form of this manuscript, or allow others to do so, for US government purposes. DOE will provide public access to these results of federally sponsored research in accordance with the DOE Public Access Plan (\url{http://energy.gov/downloads/doe-public-access-plan})}

\date{\today}

\begin{abstract}
%
  Convex \(\ell_1\) regularization using an infinite dictionary of neurons has been suggested for constructing
  neural networks with desired approximation guarantees, but can be affected by an arbitrary amount of
  over-parametrization. This can lead to a loss of sparsity and result in networks with too many
  active neurons for the given data, in particular if the \fixed{number of data samples} is large.
  As a remedy, \fixed{in this paper}, a nonconvex regularization method is investigated
  \fixed{in the context of shallow ReLU networks}: We prove that in contrast
  to the convex approach, any resulting (locally optimal) network is finite even in the
  presence of infinite data \fixed{(i.e., if the data distribution is known and the limiting
    case of infinite samples is considered)}. Moreover, we show that approximation guarantees and existing
  bounds on the network size for finite data are maintained.
\end{abstract}


\maketitle


\section{Introduction}
\label{sec:intro}


Many currently employed neural network training procedures are based on minimizing nonconvex
functionals, and employ random initialization
and stochastic optimization methods to avoid non-optimal local minima.
An alternative way to address these problems is by convex reformulation of the objective.
Here, the network architecture
can be adapted during training by gradually adding neurons, while the \fixed{output layer} weights can be
penalized by a convex sparsity-promoting functional (or regularization term). The latter \fixed{approach}
has the potential to set redundant network connections to zero during the training, which can be
subsequently removed from the network.
\fixed{Apart from reducing the network size, which improves computational efficiency of
  network evaluation, this also serves the purpose of reducing the generalization error in
  the model due to weight shrinkage.}
For early works on convex neural networks and
their theoretical analysis based on infinite feature spaces and optimization problems on
spaces of measures, we refer
to, e.g., \cite{bengio2006convex, RossetSwirszczSrebroZhu:2007, bach2017breaking}.
Additional characterizations of the associated optimal networks have recently been
provided in~\cite{savarese:2019, parhi2019minimum, debarre2020sparsest, parhi2020neural,
  2019arXiv191001635O, 2019arXiv191002743D}.

In this paper, we demonstrate that convex sparsity promoting penalties
such as the canonical \(\ell_1\) norm employed in the previous works
do not always effectively eliminate redundancy in the presence of large amounts of data, and the
associated training
procedures may still be affected by a certain amount of unnecessary
over-parametrization. As a remedy, we develop a
corresponding framework that incorporates nonconvex penalties but keeps most of the
aforementioned advantages of convex neural networks, including the measure space interpretation.
In order to do that, we focus on shallow networks with one hidden layer, which are of
fundamental importance and -- compared to deep
networks -- are relatively well understood theoretically.
Moreover, we focus attention on the activation function given as the popular ReLU function
$\sigma(z)=\max\{\,z,0\,\}$; however, most of our theory applies to a much larger class of
activation functions and also other kernel based methods.

We define a shallow neural network with $N$ neurons to be a function $\NN\colon \R^d\to \R$ of the form 
\begin{equation}\label{eq:disc_NN}
	\NN_{\omega,c}(x) = \sum_{n=1}^N c_n \, \sigma(a_n \cdot x + b_n),
\end{equation}
where \(\sigma(a_n \cdot x + b_n)\) are the single neurons and \(N\) is also
called the width of the network.
Additionally, by $\omega_n = (a_n, b_n) \in \R^{d+1}$ we denote the nodes consisting of inner weights
\(a_n \in \R^d\) and \(b_n \in \R\). The numbers $c_n\in \R$ are called the outer weights
\fixed{(output layer weights)}.
Let $f$ be a target function defined on some \fixed{set} $D\subseteq\R^d$ (it can be an
image, solution to a PDE, a specific parameter associated with a model, etc.).
The aim is to find  a neural network $\NN$ with as few as possible neurons \(N\)
such that $\NN$ fits the training points $\{(x_k,y_k)\}_{k=1,\dots,K}$. Here, $x_k\in D$
are the input data and \(y_k\) the output data, which we assume to be given as
$y_k=f(x_k)+\varepsilon_k$. 
Additionally, $\varepsilon_k$ represents error, which could be
random in nature (e.g., measurement error) or have a deterministic origin (e.g.,
modeling or computational truncation error).

The network training considered here will be based on the following
minimization problem:
\begin{equation}
\label{eq:phi_problem_finite_nonnormalized}
\min_{N \in \N,\; \{a_n,b_n,c_n\}_{n=1}^N,\;\|a_n\|^2+|b_n|^2\leq 1}\;
  l\left(\NN_{\omega,c};y\right) + \alpha\sum_{n=1}^N\phi(\abs{c_n}).
\end{equation}
For simplicity, we focus our attention on the least squares loss function
\begin{equation}
\label{eq:empirical_tracking}
l\left(\NN_{\omega,c};y\right)
 = l_{K,\fixed{x}}\left(\NN_{\omega,c};y\right)
 := \frac{1}{2K}\sum_{k=1}^K \abs{\NN_{\omega,c}(x_k) - y_k}^2,
\end{equation}
although most of our results can be transferred to a much more general class of
data fidelity terms.
We emphasize that in problem~\eqref{eq:phi_problem_finite_nonnormalized} we do not fix the
width \(N\) of the network, which will be chosen together with the corresponding
coefficients to minimize the objective.
To achieve a compromise between a good fit and a simple network (representing a hopefully
regular function with few neurons \(N\)) we employ a
sparsity promoting sublinear cost term involving the scalar penalty function~\(\phi\) for
the outer weights with hyperparameter \(\alpha>0\).
Note that \(\NN_{\omega,c}\) is linear in \(c\) and that
neurons with \(c_n = 0\) can be dropped from the network.   
For the inner weights, we impose bound constraints. This
does not limit the set of functions that can be written
as~\eqref{eq:disc_NN}, owing to the positive
homogeneity of the ReLU activation function, but prevents the inner weights from growing
arbitrarily large.
Concerning the penalty \(\phi\colon \R_+ \to \R_+\), we always impose the following assumption:
\begin{enumerate}[label=(\ensuremath{\text{A\arabic*}_\phi})]
\item \label{cond:phi_1}
  The function \(\phi\) is a concave, nondecreasing, \fixed{and differentiable} with \(\phi(0) = 0\),
  \(\phi'(0) = 1\), and \(\phi(z) \to +\infty\) for \(z\to +\infty\);
  Moreover, \(\phi\) is \(\gamma\)-convex, i.e., there exists a \(\gamma \geq 0\) such
  that the derivative \(\phi^\prime\) fulfills
  \[
    \fixed{-\gamma (z_2 - z_1) \leq \phi^\prime(z_2) - \phi^\prime(z_1) \leq 0}
    \quad\text{for all } 0 \leq z_1 \leq z_2.
  \]
\end{enumerate}
These assumptions imply that \(\phi\) fulfills \(\phi(z) \leq z\) for all
\(z \geq 0\) and that \(\phi\) is \emph{subadditive}, i.e.\ \(\phi(z_1 + z_2) \leq
\phi(z_1) + \phi(\fixed{z_2})\).
This property enhances the sparsity of the solution, and in the case that
the inequality is strict (referred to as \emph{strongly subadditive}),
it actively promotes it. We will discuss this in more detail below.
Moreover, positive homogeneity of the ReLU activation function and the monotonicity of $\phi$ also imply
that $\norm{(a_n,b_n)}=1$ will always be fulfilled for an optimal solution of the
problem~\eqref{eq:phi_problem_finite_nonnormalized}.  Consequently,
\eqref{eq:phi_problem_finite_nonnormalized} is equivalent to the following problem 
\begin{equation}
    \label{eq:phi_problem}\tag{\ensuremath{P_\phi}}
  \min_{N\in \N,\; \{c_n\} \in \R^N,\; \{\omega_n = (a_n,b_n)\} \in (\Sd)^N}\; l\left(\NN_{\omega,c};y\right) + \alpha\sum_{n=1}^N\phi(\abs{c_n}),
\end{equation}
where  $\Sd = \{(a,b)\in\R^{d+1}\colon \norm{(a,b)}^2 = 1\}$ is the unit sphere in $\R^{d+1}$.

Note that the convex \(\ell_1\) sparsity promoting penalty (i.e.\ \(\phi(z) =
z\)) is still included in the above assumptions. It is known that the problem for \(\phi(z) =
z\) can be understood as a convex problem on the space of measures, and always admits a global
solution with \(N \leq K\). Moreover, this solution can be efficiently approximated with
(generalized) conditional gradient methods (see, e.g.,
\cite{bach2017breaking,bredies2013inverse,boyd2017alternating,2019arXiv190409218P}).
While the \(\ell_1\) problem with \(\phi(z)=z\) has many favorable theoretical properties, it 
does not completely solve the issue of over-parametrization, especially in the case where
\(K\) is large.
In Figure~\ref{fig:intro}, we visualize this with a simple one-dimensional example,
where we can interpret \(\NN_{\omega,c}\) as a piecewise linear spline with knot points
\(x_n = - b_n/a_n\). We see that a nonconvex penalty term is able to substantially reduce
the number of neurons without affecting the quality of the approximation.
\begin{figure}[htbp]
\centering
\begin{subfigure}{.49\linewidth}
  \includegraphics[width=0.9\textwidth]{intro_l1.tikz}
  \caption{Global \(\ell_1\)-solution (with \(\phi(z) = z\)):\\
    \(N = 36\) neurons and \(\norm{\NN_{\omega,c} - f} = 0.044\).}
  \label{fig:intro_l1}
\end{subfigure}
\begin{subfigure}{.49\linewidth}
  \includegraphics[width=0.9\textwidth]{intro_phi.tikz}
  \caption{Local solution for \(\phi(z) = \log(1 + z)\):\\
    \(N = 10\) neurons and \(\norm{\NN_{\omega,c} - f} = 0.026\).}
  \label{fig:intro_phi}
\end{subfigure}
\caption{Comparison of solutions of~\eqref{eq:phi_problem} for \(f(x) = \cos(10(10^{-3}+x^2)^{1/8})\)
  with different convex and noncovex \(\phi\): We choose \(\alpha = 10^{-4}\), \({x_k}\)
  by 5000 uniformly distributed points on the interval \([-1,1]\), \(y_k = f(x_k)+
  \eps_k\) perturbed by white
  noise with std.\ dev.\ \(0.05\).
  \emph{Top:} outer weights \(c_n\) over the knot points \(x_n = -b_n/a_n\).
  \emph{Bottom:} Noisy data \(y_k = f(x_k) + \eps_k\) (blue), \fixed{locally} optimal
  network \(\NN_{\omega,c}\) (black), and knot points of the corresponding linear spline
  (orange).}
\label{fig:intro}
\end{figure}

To shed more light on this effect, we consider the limit of \(l_{K,\fixed{x}}\) where the
number of data points \fixed{\(K\)} grows infinitely large. If we 
interpret the data points $x_k$ to be random samples from a probability
distribution $\nu$ on the domain $D$, the loss function
$l_{K,\fixed{x}}$ can be understood as the empirical estimation of the loss
\[
l_\nu\left(\NN_{\omega,c};y\right)
 = \frac{1}{2}\int_{D} \left(\NN_{\omega,c}(x) - y(x)\right)^2\de\nu(x)
 = \frac{1}{2}\norm{\NN_{\omega,c} - y}^2_{L^2(D,\nu)},
\]
where \(y(x) = f(x) + \varepsilon(x)\), and \(\varepsilon(x)\) is an error term.
Consequently, we consider the problem~\eqref{eq:phi_problem_finite_nonnormalized} with
\(l = l_\nu\), which can be seen as approximating the function in the whole domain~\(D\)
instead of a finite number of points.

Since $N$ is maintained as a free optimization variable in~\eqref{eq:phi_problem},
one may expect that in this case the problem may not have minimizers (local or global): a
network with larger and larger number of neurons could decrease the value of the
functional. That is indeed the case for $\ell_1$ regularization as our numerical
experiments (performed for large \(K\)) indicate.
The solutions of problem~\eqref{eq:phi_problem} with the \(\ell_1\)-penalty (\(\phi(z) =
z\)) tend to form clusters of nodes  \((a_n, b_n)\) with very small coefficients
\(c_n\) (this effect is more severe for a larger number of training data~\(K\)).
The main disadvantage of the  $\ell_1$ cost functional is that it is
not encouraging nodes $(a_n,b_n)$ that are very close to merge into one. 
This is related to the additivity of the absolute value on the positive real axis: if we
replace a node $(a,b)$
with coefficient $c$  by two nodes \((a_1,b_1)\) and \((a_2,b_2)\) very closely
placed to $(a,b)$, with corresponding coefficients \(c_1\) and \(c_2\) (of the same sign as
$c$) with $c=c_1+c_2$ we obtain a network with the same \(\ell_1\)-norm, while potentially
decreasing the fidelity term \(l\).
However, switching to a nonconvex, strongly subadditive penalty,
this effect is remedied; as illustrated in Figure~\ref{fig:intro_phi}.

The above discussion highlights why the $\ell_1$ regularized problem
\begin{equation} \tag{\ensuremath{P_{\ell_1}}}
   \label{eq:l1_problem}
    \min_{N\in \N,\; \{c_n\} \in \R^N,\; \{\omega_n=(a_n,b_n)\} \in (\Sd)^N}\; l\left(\NN_{\omega,c};y\right) + \alpha\sum_{n=1}^N\abs{c_n}
\end{equation}
is not the best choice for sparsity promoting regularization of neural networks.
Certainly, this also affects formulations where we replace the penalty term in~\eqref{eq:l1_problem} with a
constraint \(\norm{c}_{\ell_1} \leq M\), or the fidelity term with a constraint
\(\norm{\NN_{\omega,c} - f} \leq \delta\), since they essentially lead to the same solution
manifolds (parametrized by different hyperparameters \(\alpha\), \(M\), and \(\delta\)).
Moreover, we point out that the global solutions of the problem~\eqref{eq:l1_problem} are also global solutions of
the popular formulation below, employing \(\ell^2\) regularization: 
\begin{equation}
  \label{eq:naive_l2}\tag{\ensuremath{P_{\ell_2,\ell_2}}}
  \min_{N\in \N,\; \{(a_n,b_n,c_n)\} \in (\R^d\times\R\times\R)^N}\; l\left(\NN_{\omega,c};y\right)
  + \frac{\alpha}{2}\sum_{n=1}^N\left[\norm{(a_n,b_n)}^2 + \abs{c_n}^2\right].
\end{equation}
This equivalence relies on the positive homogeneity of the ReLU activation function;
see, e.g.,~\cite{2014arXiv1412.6614N}.
Thus,~\eqref{eq:naive_l2} is surprisingly already equivalent to a sparsity
regularized problem in this setting, and is also affected by the same issues
as~\eqref{eq:l1_problem}.

More generally, if we replace the cost term in~\eqref{eq:naive_l2} with
\(\mathcal{R}(a,b,c) = (1/p) \sum_n \norm{(a_n,b_n)}^p + \abs{c_n}^p\), we obtain the
problem formulation~\eqref{eq:phi_problem} with \(\phi(z) = (2/p)\, z^{p/2}\) (see
Appendix ~\ref{sec:equiv_cond}).
For \(0<p<2\), the choice of  \(\phi(z) = (2/p)\, z^{p/2}\) is concave,
monotonous, and subadditive, and would be appropriate for some parts of this paper.
However, it does not fulfill the other requirements~\ref{cond:phi_1} imposed above,
since it has an unbounded derivative at zero and can not be normalized to fulfill
\(\phi'(0) = 1\).
Since this assumption is crucial for this paper, we instead consider
a \emph{strongly subadditive} function $\phi \colon \R_+ \to \R_+$
fulfilling~\ref{cond:phi_1} and the following additional assumption:
\begin{enumerate}[resume, label=(\ensuremath{\text{A\arabic*}_\phi})]
\item \label{cond:phi_2}
  There exists a \(\widehat{\gamma} > 0\) and \(\widehat{z} > 0\) such that
  \[
    \fixed{\phi'(z_2) - \phi'(z_1) \leq -\widehat{\gamma} (z_2 - z_1)}
    \quad\text{for all } 0 \leq z_1 \leq z_2 \leq \widehat{z}.
  \]
\end{enumerate}
It can be observed that any $\phi$ possessing the properties~\ref{cond:phi_1}
and~\ref{cond:phi_2} also satisfies the inequalities
$z - (\gamma/2) z^2 \leq \phi(z) \leq z - (\widehat{\gamma}/2) z^2$  for $z \in [0,
\widehat{z}]$.
The function
\begin{equation}\label{form:phi_log}
  \phi_{\log,\gamma}(z)
  = \frac{1}{\gamma} \log(1+\gamma z)
  \fixed{
    = \int_0^z \frac{1}{1+\gamma \zeta} \de\zeta
  }
\end{equation}
for \(\gamma > 0\),
which is a scaled version of the log-penalty function (considered in, e.g.,~\cite{mazumder2011sparsenet}),
and its convex combination with the $\ell_1$ norm,
will be the main function of choice for us in the numerical examples. Another option is
the MCP function~\cite{zhang2010nearly},
\[
  \operatorname{MCP}_\gamma(z)
  =
  \fixed{\int_0^z \max\{\,0,\, 1-\gamma \zeta\,\} \de\zeta}
  =
  \begin{cases}
    z - (\gamma/2) z^2 &\text{for } z < 1/\gamma, \\
    1/(2\gamma) &\text{else,}
  \end{cases}
\]
\fixed{which fulfills~\ref{cond:phi_2} for \(\gamma > 0\)}, but lacks the property \(\phi(z) \to +\infty\) for \(z\to+\infty\).
However, a proper convex combination of MCP with \(\ell_1\), e.g.\
\(\phi(z) = (1/2)(z + \operatorname{MCP}_{2\gamma}(z))\), fulfills both~\ref{cond:phi_1}
and~\ref{cond:phi_2}.
\fixed{
We discuss the former choices in the context of the general assumptions in Appendix~\ref{sec:phi_reg}.
}
\fixed{
Another penalty function, the SCAD penalty function, given by
\[
  \operatorname{SCAD}_{\gamma,\lambda}(z)
  =
  \fixed{\int_0^z \min\left\{1,\,\max\{0,\, 1-\gamma(\zeta - \lambda)\}\right\} \de\zeta}
  =
  \begin{cases}
    z   &\text{for } z < \lambda \\
    z - (\gamma/2) (z-\lambda)^2   &\text{for } \lambda \leq z < \lambda + 1/\gamma, \\
    \lambda + 1/(2\gamma)  &\text{else,}
\end{cases}
\]
where \(\lambda > 0\) and \(\gamma > 0\),
does not fulfill~\ref{cond:phi_2} due to being linear close to \(z \approx 0\).}
\begin{figure}[htbp]
\centering
\tikzexternaldisable
\begin{tikzpicture}[
  declare function={
    scad(\x) = (\x < 1) * (\x) +
               and (\x >= 1, \x < 2) * (1.5-(\x-2)^2/2) +
               (\x >= 2) * (1.5)
               ;
    mcp(\x) = (\x < 1) * (\x - (\x)^2/2)   +
              (\x >= 1) * (1/2)
              ;
    logp(\x) = ln(1 + \x)
    ;
    lone(\x) = \x
    ;
    scad_p(\x) = (\x < 1) * (1) +
               and (\x >= 1, \x < 2) * ( -(\x-2)) +
               (\x >= 2) * (0)
               ;
    mcp_p(\x) = (\x < 1) * ( -(\x-1))   +
              (\x >= 1) * (0)
              ;
    logp_p(\x) = 1/(1 + \x)
    ;
    lone_p(\x) = 1
    ;
    ratp(\x) = \x/(1 + \x)
    ;
    ratp_p(\x) = 1/(1 + \x)^2
    ;
  }
]
\begin{axis}[%
    width=3in,
    height=2.5in,
    at={(0in,0in)},
    axis x line=middle,
    axis y line=middle,
    ymin=0, ymax=3,
    ytick={0,0.5,...,3}, ylabel=$y$,
    xmin=, xmax=3,
    xtick={0,0.5,...,3}, xlabel=$x$,
    domain=0:3,samples=201, 
    legend pos=north west,
]

  \addplot [gray,thick] {lone(x)};
  \addlegendentry{$\phi(z) = z$}
  \addplot [black,thick] {logp(x)};
  \addlegendentry{$\phi_{\log,1}$}
  \addplot [black,dashed] {mcp(x)};
  \addlegendentry{MCP$_{1}$}
  \addplot [gray,dashed] {scad(x)};
  \addlegendentry{SCAD$_{1,1}$}

\end{axis}
\begin{axis}[
    width=3in,
    height=2.5in,
    at={(3in,0in)},
    axis x line=middle,
    axis y line=middle,
    ymin=0, ymax=2,
    ytick={0,0.5,...,2}, ylabel=$y$,
    xmin=, xmax=3,
    xtick={0,0.5,...,3}, xlabel=$x$,
    domain=0:3,samples=201, 
]
\addplot [gray,thick] {lone_p(x)};
  \addlegendentry{$\phi^\prime(z) = 1$}
\addplot [black,thick] {logp_p(x)};
  \addlegendentry{$\phi_{\log,1}^\prime$}
\addplot [black,dashed] {mcp_p(x)};
  \addlegendentry{MCP$^\prime_1$}
\addplot [gray,dashed] {scad_p(x)};
  \addlegendentry{SCAD$^\prime_{1,1}$}

\end{axis}
\end{tikzpicture}
\tikzexternalenable
\caption{Comparison of different penalty functions \(\phi\)
  and their derivatives \(\phi'\).
  The \(\ell_1\) and log-penalty fulfill~\ref{cond:phi_1}, MCP and SCAD
  fulfill~\ref{cond:phi_1} aside from \(\phi(z) \to +\infty\) for \(z\to\infty\).
  The log-penalty and MCP also fulfill~\ref{cond:phi_2}.}\label{fig:intro_plot_phi}
\end{figure}
We refer to Figure~\ref{fig:intro_plot_phi} for a visualization of different penalty
functions.

\subsection{Contribution}
For $\phi$ satisfying the conditions~\ref{cond:phi_1} and~\ref{cond:phi_2}
we show that the problem \eqref{eq:phi_problem} has
global and local minimizers (with finite \(N\)) for finite and infinite data (see
Theorem~\ref{thm:local_finiteness}).
Since the existence of finite minimizers of the convex problem relies on the finite
data, and this is not the case for the nonconvex problem,
this is a rather unexpected result.

Since the regularization term in \eqref{eq:phi_problem} is nonconvex, finding the global solution of the problem may
not be feasible. In fact, for a nonconvex optimization problem of similar structure,
finding global minima with very high precision is shown
to be an NP-hard problem; see~\cite{chen2019approximation}.  On the other hand, it is
observed in practice and confirmed in theory (see, e.g., \cite{loh2015regularized}) that local
minima (or stationary points, in general) of nonconvex regularized 
problems tend to be well behaved.
We will develop a similar theory for the problem~\eqref{eq:phi_problem}, where \(N\) is a
free optimization variable. First, we will define a concept of local minimality, which is
based on the notion of locality in the space of shallow neural networks of the form
\eqref{eq:disc_NN} defined in terms of the associated measure
\begin{equation}
  \label{eq:discrete_measure}
  \mu = \sum_{n=1}^N c_n \, \delta_{(a_n,b_n)},
\end{equation}
where $\delta_{(a_n,b_n)}$ is the Dirac measure at $(a_n,b_n)$.
We will refer to these solutions as ``local solutions in the sense of measures'',
which will be further explained in the context of integral neural networks
discussed in Section~\ref{sec:theory}. We only mention that a local solution in our
setting will be any \fixed{measure~\eqref{eq:discrete_measure}, where the outer weights \(c_n\) are minimal on in a suitable
neighborhood}, and where adding to \(\NN_{\omega,c}\) any additional node \((a,b) \in \Sd\)
with suitably small outer weight \(c\) will also increase the training objective; see
Theorem~\ref{thm:nec_suff_local_sol}.

For these local solutions in the sense of measures of~\eqref{eq:phi_problem} (which also include the global
solutions), we show that:
\begin{itemize}
  \item
    They are always finitely supported; see Theorem~\ref{thm:local_finiteness}.
  \item
    In the case \(K < \infty\) it holds \(N \leq K\); see Theorem~\ref{thm:representer_phi}.
  \item
    The approximation error can be estimated by
    \begin{align}
      \label{eq:est_loc_sol}
      l(\NN_{\omega,c};f) \leq  2\, \alpha \fixed{\norm{f}_{\mathcal{W}(D)}} + l(y;f)
    \end{align}
    \fixed{for any \(f \in \mathcal{W}(D)\), as introduced in Section~\ref{subsec:local_sol}}; see~Theorem~\ref{thm:fid_est}.
\end{itemize}
The last point quantitively affirms the assertion that the hyperparameter $\alpha$ can be treated  as a trade-off
between the network sparsity and reconstruction accuracy \fixed{(at least for a
  well-behaved subclass of functions \(f\))}. In particular, it shows that local
solutions in the sense of measures of~\eqref{eq:phi_problem} can reduce the fitting term
\fixed{to a similar level as the level of
noise (or bias)} in the data by an appropriate choice of \(\alpha\), i.e., \fixed{\(\alpha
\leq C \, l(y;f)\)}.

Finally, we propose a method to
algorithmically approximate local solutions  in the sense of measures of~\eqref{eq:phi_problem}. This is based upon
an extension of the methods developed for \(\ell_1\)
regularization~\cite{bredies2013inverse,boyd2017alternating,bach2017breaking,2019arXiv190409218P,flinth2019linear}
to the setting of nonconvex penalties \(\phi\).
Here, we combine adaptive node insertion and deletion with local minimization of the outer
weights (or, optionally, full gradient-based training of all the weights \((a,b,c)\)).
Again, we rely on the property \(\phi'(0) = 1\) and the corresponding optimality
conditions to guide the node insertion and deletion steps. 

\subsection{Limitations}
\fixed{
The deterministic error estimate is one of main points that justify the presented solution
concept. Let us briefly mention several shortcomings of the presented approach and discuss
possible extensions, which are, however, outside of the scope of this paper.
\begin{itemize}
\item
  The bound~\eqref{eq:est_loc_sol} is only applicable to well-behaved (regular) functions in the space \(\mathcal{W}(D)\), which is
  endowed with a norm that is sometimes referred to as the variation of
  \(f = \NN_{\omega,c}\); cf., e.g., \cite{bach2017breaking,klusowski2016risk}, the
  references therein, and Section~\ref{sec:exact_rep} for a detailed discussion.
  Since the variation is defined here in the canonical way (independent of the functional
  \(\phi\)), we can apply known results to extend the bound to less regular functions
  \(g\) (e.g., Lipschitz-continuous).
  We first approximate~\(g\) by a function~\(g_\varepsilon \in \mathcal{W}(D)\)
  with \(l(g;g_\varepsilon) \leq \varepsilon^2\); cf., e.g.,
  \cite[Section~4.3]{bach2017breaking}. Then we can use~\eqref{eq:est_loc_sol} with
  \(f=g_\varepsilon\) and apply the triangle inequality for an estimate with \(g\).
\item
  The error of the fit in~\eqref{eq:est_loc_sol} is defined in terms of the functional
  used for training. In practice, the exact distribution \(\nu\) is usually not
  available or computationally accessible, and training is performed on a sample, using
  \(l_{K,x}\) for noisy data \(y\). However, bounds for the error \(\NN_{\omega,c}-f\) are
  desired for every \(x \in D\). This problem of generalization can be addressed by
  combining~\eqref{eq:est_loc_sol} with estimation bounds for
  \(l_\nu(f;y) - l_{K,x}(f,y)\) over all \(\norm{f}_{\mathcal{W}(D)} \leq \delta\); see,
  e.g., \cite[Section~5]{bach2017breaking}.
\item
  The property \(\phi'(0) = 1\), which is fundamentally required
  for~\eqref{eq:est_loc_sol}, could be replaced by \(\phi'(0) <
  +\infty\) with minor modifications.
  However, as mentioned before, Assumption~\ref{cond:phi_1} excludes the ``\(q\)-norms'' \(\phi(z) = (1/q)\,z^q\) with
  \(q<1\) since here \(\phi'(0) = +\infty\).
  In this case, inserting additional neurons with sufficiently small outer
  weights \(c\) will always increase the objective (due to \(\phi'(0) = +\infty\)) and
  thus the zero measure is always a local solution in the sense of measures. Thus, no
  approximation guarantees can be given for arbitrary local solutions.
  Certainly, there can be local solutions that fit the data properly, but finding them
  algorithmically with gradient based optimization has to rely on an appropriate
  initialization.
  In particular, the node insertion strategy of the algorithm presented in
  Section~\ref{sec:gen_conj_grad} relies on~\(\phi'(0) = 1\) and can
  not be directly adapted to \(\phi(z) = (1/q)\,z^q\).
\end{itemize}
}

\subsection{Related work}
Sparsity has been widely employed with the dual purpose of removing non-informative
connections from neural networks~\cite{2019arXiv190610732E, 2019arXiv190209574G} and also
to guide adaptive architecture search~\cite{bengio2006convex, bach2017breaking, cortes2017adanet};
in our case it is the adaptive choice of the network width.

Training procedures with nonconvex penalties have been employed in order to eliminate certain weights from the network. 
In~\cite{2019arXiv190101021M} a nonconvex penalty with
$\phi(z) = (\beta+1) z / (\beta + z)$ is proposed.
Note, that a rescaled version of this penalty fulfills all the
requirements of our analysis.
Similarly, in~\cite{yang2019deephoyer} a different nonconvex regularization strategy is proposed
that is based on the ratio of $\ell_1$ and $\ell_2$ norms and does not have
separable form as we consider in~\eqref{eq:phi_problem}.
However, unlike discussed here, these works apply the penalty to all the weights
in the network with a fixed architecture and do not consider \(N\) to be variable.
Moreover, the nonconvex regularization approaches proposed above \fixed{rely on random
initialization to find a ``good'' local minimum by local optimization} and do not lead to any
deterministic approximation guarantees, which \fixed{is a general problem known to negatively} affect training
procedures based on random initialization and local all-weight training
(see, e.g., \cite{2020arXiv200107523A}).

Nonconvex penalties for sparse regularization have been considered in the statistics
literature. The functions like SCAD, the MCP and capped $\ell_1$ are popular choices
\cite{fan2001variable,zhang2010analysis,zhang2010nearly,wang2014optimal}.
We remark that the MCP penalty fulfills most of the conditions of our analysis;
cf.\ Figure~\ref{fig:intro_plot_phi}.
However, the dictionary and the data set can be
infinite in our work, and the existing results do not directly apply to the
neural network model we consider here.

Nonconvex functionals on spaces of measures also appear in the study of
gradient regularization, such as problems involving functions of
bounded variation (TV-norm of the gradient).
In fact, such problems initially prompted the characterization of lower-semicontinuity of
nonconvex functionals of measures~\cite{bouchitte1990new,bouchitte1992integral,bouchitte1993relax}, which is
also used for the existence theory provided here.
However, the gradient of a function can only have atomic parts (Dirac delta functions)
in one spatial dimension, and  not in higher spatial dimensions.
Therefore, nonconvex regularizers of the gradient face additional
challenges~\cite{HVW:2015}, and are usually implemented after
discretizing the problem.
We note that the aforementioned works \fixed{do not provide the results from
Section~\ref{sec:theory} on finitely supported minima of optimal solutions under the
additional condition~\ref{cond:phi_2}}.

Integral neural networks have been around for a while and there is a large volume of
work in this direction; see, e.g.,
\cite{bach2017breaking,2019arXiv191001635O}.  Integral neural networks have, in
particular, been used for demonstrating approximation capabilities of shallow neural
networks \cite{barron1993universal,kuurkova1997estimates,ma2019barron}. Under the integral
representation assumption, it is possible to prove certain convergence rates for greedy type
algorithms. Other related work includes integral neural network representation results
like those in  \cite{kainen2010integral,klusowski2018approximation} and the ridgelet
transform \cite{candes1999harmonic, murata1996integral, sonoda2017neural}.

\subsection{Organization}

In Section~\ref{sec:theory}, we develop the main theoretical framework as outlined
above. We introduce integral neural networks, and
extend problems \eqref{eq:l1_problem} and \eqref{eq:phi_problem} to this framework.
Moreover we state the main contributions concerning various
properties of local solutions  in the sense of measures, their existence, necessary
conditions, finiteness, and good fidelity.
Section~\ref{sec:exact_rep} is about the least total variational norm solution of the exact
representation constrained problem. This problem is strongly connected to the fidelity
estimate in Theorem~\ref{thm:fid_est}.
In Section~\ref{sec:gen_conj_grad} we propose \fixed{an} algorithmic solution method,
which is an adaptation of the generalized \fixed{conditional} gradient
method to the nonconvex setting.
Finally, in Section~\ref{sec:numerics}, we illustrate the \fixed{results}
of this paper with concrete examples in one \fixed{to three} dimensions. 
The paper ends with an Appendix, where we provide proofs of the
theorems from Section~\ref{subsec:local_sol}.


\section{General theory}
\label{sec:theory}

Problem~\eqref{eq:phi_problem} is a particular case of a more general framework that will
be discussed in this section. We will make use of the concept of integral neural \fixed{networks}
which is a generalization of the neural network~\eqref{eq:disc_NN} where the sum in
$\NN_{\omega,c}(x)$ is replaced with an integral. We will introduce the extensions of
problems~\eqref{eq:phi_problem} and~\eqref{eq:l1_problem} for integral neural networks,
and obtain various analytic results concerning their local solutions.  Most of the results
here apply to a larger class of activation functions than just the ReLU so the theory will
be developed in this setting.

\subsection{Integral neural networks}\label{subsec:int_neural}

Let $\Omega$ be a compact subset of \fixed{$\R^{d+1}$}. Denote by $M(\Omega)$ the space of
\fixed{real-valued} Borel measures on $\Omega$ of bounded total variation,
and by $\norm{\mu}_{M(\Omega)}$ the total
variation norm of the measure $\mu\in M(\Omega)$.
Consider a \fixed{set \(D \subset \R^d\) and a} function
$\sigma\in C(D\times\Omega)$ such that for some \(\Lambda>0\) it holds
\begin{align}
  \label{eq:Lipschitz_sigma}
  \abs{\sigma(x;\omega_1)-\sigma(x;\omega_2)}
  &\leq\Lambda \fixed{\left(1+\norm{x}\right)} \norm{\omega_1-\omega_2}
  &&\fixed{\text{for all } x\in D, \omega_1,\omega_2 \in \Omega},
\end{align}
\fixed{and \(\abs{\sigma(x;\omega_0)} \leq \Lambda \fixed{(1+\norm{x})}\)
  for some \(\omega_0 \in \Omega\) and all \(x\in D\).}
Note that with the ReLU activation function,
i.e.\ \(\sigma(x;\omega)=\max\{\,a\cdot x+b,\,0\,\}\),
this condition is satisfied with $\Lambda=1$ \fixed{for \(\omega = (a,b) \in \Omega = \Sd\)}. 
An integral neural network is a function of the form
\[
[\NN\mu](x) = \int_{\Omega} \sigma(x;\omega) \de\mu(\omega)
\]
where $\mu\in M(\Omega)$.
Finally, by
\[
\pair{\varphi, \mu} = \int_{\Omega} \varphi(\omega) \de \mu(\omega),
\]
we denote the canonical duality pairing of \(\varphi \in C(\Omega)\) and \(\mu \in
M(\Omega) = C(\Omega)^*\).

To see that the above integral network is an extension of the network in
\eqref{eq:disc_NN}, let $\Omega=\Sd$, $\sigma(x;\omega)=\sigma(a\cdot x+b)$ for
$\omega=(a,b)\in\Sd$, and define the discrete measure
\begin{equation}\label{eq:measure_discrete}
  \mu = \sum_{n=1}^N c_n \delta_{\omega_n} \in M(\Omega).
\end{equation}
Then it can be observed that 
\begin{equation}\label{eq:measure_NN_identification}
  [\NN\mu](x) = \sum_{n=1}^N c_n \sigma(a_n\cdot x+b_n) = \NN_{\omega,c}(x).
\end{equation}
Additionally, it holds that 
\[
  \norm{\mu}_{M(\Omega)} = \sum_{n=1}^N|c_n| = \norm{c}_{\ell_1},
\quad
\pair{\varphi,\mu} = \sum_{n=1}^N \varphi(\omega_n) c_n = \inner{\varphi(\omega), c}_{\R^N},
\]
which relates the total variation norm of \(\mu\) to the \(\ell_1\) norm of \(c\) and
the duality pairing to an Euclidean inner product of the vector
\((\varphi(\omega_n))_n \in \R^N\) with \(c\).

Now, we turn to the loss function.
Let $\nu$ be a probability measure supported on the set $D=\fixed{\supp(\nu)}\subseteq \R^d$ with finite first and
second moments; i.e.\ \(\int_D \fixed{\norm{x}}^2 \de\nu(x) = \norm{x}^2_{L^2(D,\nu)} < \infty\).
Associated to this, we define the Hilbert space \( L^2(D,\nu)\) of square integrable
functions with respect to \(\nu\).
The last property ensures that \(\NN\) is bounded
as an operator from \(M(\Omega)\) to \( L^2(D,\nu)\), where we let
\[
\norm{\NN} = \max_{\omega\in \Omega}\norm{\sigma(\cdot;\omega)}_{L^2(D,\nu)} < \infty
\]
be its operator norm.
\fixed{
Note that \(\norm{\sigma(\cdot;\omega)}_{L^2(D,\nu)}
\leq \norm{\sigma(\cdot;\omega) -
\sigma(\cdot;\omega_0)}_{L^2(D,\nu)} + \norm{\sigma(\cdot;\omega_0)}_{L^2(D,\nu)}\), and
both terms can be bounded using using~\eqref{eq:Lipschitz_sigma} by \(C \Lambda (1 +
\norm{x}_{L^2(D,\nu)})\), where \(C\) only depends on the compact set \(\Omega\).}
Moreover, let $f\in \Hi$
be the target function we aim to approximate with integral neural networks. From the
observation above, \fixed{it can be seen that a consistent} extension of the
$\ell_1$-regularized problem~\eqref{eq:l1_problem} \fixed{from finitely supported to arbitrary measures is given} as
\begin{equation}
  \label{eq:L1_problem}\tag{\ensuremath{P_{L^1}}}
  \min_{\mu \in M(\Omega)} L(\mu) + \alpha\, \norm{\mu}_{M(\Omega)}
\end{equation}
where 
\[
  L(\mu) = l(\NN \mu;y) =  \frac{1}{2} \norm*{\NN\mu-y}^2_{\Hi}.
\]
Here, \(y = f + \varepsilon \in \Hi\) is a potentially biased or noisy
version of function \(f\) to be approximated.
We remark that the empirical functional \(l_K\) from~\eqref{eq:empirical_tracking} is still included in this formulation
by the choice \(\nu = \nu_K = (1/K) \sum_{k} \delta_{x_k}\) and identifying \(y \in
L^2(D,\nu_K)\) with the vector \(y_k = y(x_k) = f(x_k) + \varepsilon(x_k) = f(x_k) + \varepsilon_k\).

To extend the problem~\eqref{eq:phi_problem} for integral neural networks,  we need to
define the analog  of the penalty term in~\eqref{eq:phi_problem} for arbitrary measures
\(\mu\in M(\Omega)\).
To do this, we first recall that any finite measure \(\mu\) can be uniquely decomposed into
an atomic part, which is a (potentially infinite) sum of Dirac-delta measures, and the
remaining continuous part.
Denote by \(\atom \mu = \{\,\omega_n\,\}_n\) the atoms of \(\mu\in M(\Omega)\), of which
there are either a finite number or countably infinitely many.
Then we define 
\begin{equation}
  \label{eq:def_Phi}
    \Phi(\mu) =
    \abs{\mu}\left(\Omega\setminus \atom\mu\right) + \sum_{n} \phi(\abs{\mu}(\{\,\omega_n\,\})),
\end{equation}
where $\abs{\mu}$ is the total variation measure of $\mu$.
\fixed{
Defining the decomposition into atomic and continuous part \(\mu = \mu_{\text{atom}} + \mu_{\text{cont}}\), where
\(\mu_{\text{atom}} = \sum_n c_n \delta_{\omega_n}\) with \(c_n = \mu(\{\,\omega_n\,\})\)
and \(\mu_{\text{cont}} = \mu \rvert_{\Omega \setminus \atom \mu}\),
we can equivalently write
\[
\Phi(\mu) = \norm{\mu_{\text{cont}}}_{M(\Omega)} + \sum_{n} \phi(\abs{c_n}).
\]
}
We note that this functional is weakly lower semicontinuous with respect to weak-\(*\)
convergence (see~\cite[\fixed{Theorem~3.3}]{bouchitte1990new}), which will be important
in the following.
\fixed{Here, the assumption \(\phi'(0)=1\) is essential, otherwise the first term in the
  definition of \(\Phi\) would need to be multiplied by \(\phi'(0)\).}
Moreover, \(\Phi\) is identical to \fixed{the weakly-\(*\) lower semicontinuous envelope
of \(\Phi_{\text{atom}}(\mu) = \sum_{\omega \in \atom\mu} \phi(\abs{\mu}(\{\,\omega\,\}))\);
see~\cite[Theorem~3.2]{bouchitte1993relax}.}
\fixed{Clearly, \(\Phi(\mu) = \norm{\mu}_{M(\Omega)}\) for a
measure with no atoms  and in general the following inequality holds
\[
\phi(\norm{\mu}_{M(\Omega)}) \leq \Phi(\mu) \leq \norm{\mu}_{M(\Omega)},
\]
using the sub-additivity of \(\phi\);
see Lemma~\ref{lem:phi_ineq} in Appendix~\ref{sec:phi_reg}.
}
As the canonical generalization of the problem \eqref{eq:phi_problem}, we then consider the following problem 
\begin{equation}
  \label{eq:Phi_problem}\tag{\ensuremath{P_{\Phi}}}
  \min_{\mu \in M(\Omega)} L(\mu) + \alpha\, \Phi(\mu).
\end{equation}
One advantage we gain from expanding the definition of finite width neural networks to
infinite width neural networks is that measures come equipped with  the  total variation norm topology
that allows us to define a concept of locality in a straightforward way.

\subsection{Local solutions of the $\phi$-regularized problem} \label{subsec:local_sol}

Finding the global solution of the nonconvex problem~\eqref{eq:Phi_problem} may not be
realistic, so instead, we will investigate its local minima, and show that they possess
desirable properties. The first result is to show that minima of the functional
\begin{equation}\label{eq:J-func}
J(\mu) = L(\mu) + \alpha\, \Phi(\mu)
\end{equation}
in fact exist. To this purpose, we introduce an appropriate notion of a local minimum.

\begin{definition}\label{def:local_min_meas}
\(\bar{\mu} \in M(\Omega)\) is a local minimum if there exists an \(\epsilon>0\), such
that
\[
  \fixed{J(\mu) \geq J(\bar{\mu})}\quad\text{for all } \mu \in M(\Omega)
\text{ with } \norm{\mu - \bar{\mu}}_{M(\Omega)} \leq \epsilon.
\]
\end{definition}

The next theorem establishes the existence of minimizers under the minimal
assumptions~\ref{cond:phi_1}, which also cover the $\ell_1$-penalty function.

\begin{theorem}\label{thm:local_attainable}
  If $\phi$ satisfies conditions~\ref{cond:phi_1}, then~\eqref{eq:J-func}
  admits at least one global minimizer \fixed{in \(M(\Omega)\) which is bounded in terms of the
  data:}
  \[
    \fixed{
      \phi\left(\norm{\mu}_{M(\Omega)}\right) \leq \frac{l(0;y)}{\alpha}.
    }
  \]
\end{theorem}
\fixed{The proof is given in Appendix~\ref{sec:phi_reg}.}
Note, that the global minimum is also a local minimum; however, the optimization algorithm
that we employ in practice can only approximate a local minimum.
Moreover, the local and global solutions of~\eqref{eq:Phi_problem} only correspond to
solutions of~\eqref{eq:phi_problem} in a generalized sense; thus far, there is no
guarantee that the solutions are discrete, i.e.\ have a
representation as in~\eqref{eq:measure_discrete}.
To analyze this, we first derive first-order conditions for the local solutions.

To characterize local
solutions of~\eqref{eq:Phi_problem} with a first-order necessary
condition, we require some additional notation. Denote by
\(\nabla L(\mu)\) the gradient of the loss function~\(L\), which is defined as
\[
\pair{\nabla L(\mu), u}
 = \lim_{\tau \to 0} (1/\tau) \left[L(\fixed{\mu} + \tau\,u) - L(\fixed{\mu})\right],\; \forall u\in   M(\Omega).
\]
It  holds that \(\nabla L(\mu) = \NN^* \nabla l(\NN \mu;y) =  \NN^* (\NN\mu - y)\in C(\Omega)\), where 
\[
\NN^*\colon \Hi \to C(\Omega),\quad [\NN^*g](\omega) = \int_{D} \sigma(x;\omega) g(x) \de \nu(x)
\quad\forall g \in \Hi
\]
is the (pre-)adjoint of \(\NN\). The gradient 
$p = \NN^* (\NN\mu - y)\in C(\Omega)$ will be called the dual variable in the
following. We note that the dual variable gives the inner product of the residual
\(\NN\mu - y\) with \(\sigma(\cdot; \omega)\), i.e.
\[
\bar{p}(\omega) = \int_D \sigma(x;\omega) [\NN\mu - y](x) \de\nu(x)
= \inner{\sigma(\cdot;\omega), \NN\mu - y}_{\Hi}.
\]
It serves to characterize the local solutions of~\eqref{eq:Phi_problem} as follows.
\begin{theorem}\label{thm:opt_cond}
Let \(\phi\) fulfill the conditions~\ref{cond:phi_1}.
If \(\bar{\mu}\) is a local minimum of functional $J$, then the optimal dual variable
\(\bar{p} = \nabla L (\bar{\mu}) = \NN^* (\NN\bar{\mu} - y) \in C(\Omega)\) has the following properties:
\begin{align*}
|\bar{p}(\omega)| &\leq \alpha && \text{for } \omega \in \Omega, \\
\bar{p}(\omega) &= - \alpha \, \phi'(\abs{\bar{\mu}(\{\,\omega\,\})}) \, \sign(\bar{\mu})(\omega)
 &&\text{for } \fixed{\bar\mu\text{-almost all }}\omega \in \supp \bar{\mu}.
\end{align*}
Here, \(\sign(\bar\mu) \colon \supp\bar\mu \to \{-1,1\}\) denotes the signum of \fixed{\(\bar\mu\)},
defined $\bar\mu$-a.e.\ uniquely for \(\omega\in\supp\bar \mu\) (by the Hahn decomposition).
\end{theorem}
We refer to Appendix~\ref{sec:opt_cond} for the proof of this result. We note that this
only gives a necessary condition for optimality, and that the interpretation of the second
condition requires abstract tools from measure theory.

In the previous results, we still include the case \(\phi(z) = z\) corresponding
to~\eqref{eq:L1_problem}. In this situation, we can not guarantee that a solution
of the form~\eqref{eq:measure_discrete} exists (which is also evidenced by our numerical
experiments in Section~\ref{sec:numerics}).
However, if~\ref{cond:phi_2} holds, we derive that the local solutions
to~\eqref{eq:Phi_problem} are finitely supported and thus of the form~\eqref{eq:measure_discrete}.

\begin{theorem}\label{thm:local_finiteness} 
Suppose $\phi$ satisfies conditions~\ref{cond:phi_1} and~\ref{cond:phi_2}.
If \(\bar{\mu}\) is a local solution of~\eqref{eq:Phi_problem}, then there exists
\(N<\infty\), \(\bar{\omega}_n\) and corresponding coefficients
\(\bar{c}_n \neq 0\), \(n = 1,\ldots,N\), with
\( \bar{\mu} = \fixed{\sum_{n=1}^N} \bar{c}_n \delta_{\bar{\omega}_n}\).
\end{theorem}
\fixed{The proof of this result is given in Appendix~\ref{sec:local_finiteness}.}
In the case of an atomic local minimum, the necessary optimality conditions from
Theorem~\ref{thm:opt_cond} can be further simplified.
\begin{remark}\label{rem:finite}
Let \(\phi\) fulfill the conditions~\ref{cond:phi_1}.
Let $\bar{\mu}$ be a \fixed{finitely supported} local minimum of \(J\) (as in Theorem~\ref{thm:local_finiteness}) with \(\bar{c}_n \neq 0\). Then, the second condition of Theorem~\ref{thm:opt_cond} reads as
\begin{equation}\label{eq:necess_local_min_outer}
\bar{p}(\bar{\omega}_n) = - \alpha\,\phi'(\abs{\bar{c}_n})\sign\bar{c}_n ,\; n = 1,\ldots,N,
\end{equation}
where the dual variable is \(\bar{p} = \NN^*(\NN_{\bar{\omega},\bar{c}} - y)\).
For the optimal \(\bar{\omega}\) define also the functional
\[
  J_{\bar{\omega}}(c) = l(\NN_{\bar{\omega},c};y) + \alpha \sum_{n=1}^N \phi(\abs{c_n}).
\]
It can be seen that
condition \eqref{eq:necess_local_min_outer} is the first order necessary condition of local
optimality of $\bar{c}$ being a local minimizer of \(J_{\bar{\omega}}\).
\fixed{
  This follows from
  \(0 = \partial_{c_n} J_{\bar{\omega}}(\bar{c}_n)
  = \bar{p}(\bar{\omega}_n) + \alpha\,\phi'(\abs{\bar{c}_n})\sign\bar{c}_n\) employing
  straightforward calculations.}
\end{remark}
We note that with~\ref{cond:phi_2}, due to \(\phi'(z) \leq \fixed{1 - \widehat{\gamma} z}\), these
conditions imply that \(\abs{\bar{p}(\bar{\omega}_n)} < \alpha\) for all
\(n = 1,\ldots,N\), in addition to \(\abs{\bar{p}} \leq \alpha\), which holds uniformly on \(\Omega\).
However, for a nonconvex problem, the necessary conditions above are not sufficient for optimality.
The next theorem provides a slightly stronger condition that
turns out to be sufficient for local optimality.
\begin{theorem}\label{thm:nec_suff_local_sol}
Let \(\phi\) fulfill the conditions~\ref{cond:phi_1} and~\ref{cond:phi_2}.
Let $\bar{\mu} = \sum_{n=1}^N \bar{c}_n \delta_{\bar{\omega}_n}$ be a \fixed{finitely supported} measure such that:
\begin{itemize}
\item[i)]
  \(\bar{c} \in (\R\fixed{\setminus\{0\}})^N\) is a local minimum of \(J_{\bar{\omega}}\).
\item[ii)]
  For all \(\omega \in \Omega \setminus \{\,\bar{\omega}_n\,\}_{n=1,\ldots,N}\) it holds 
   \(\abs{\bar p(\omega)} < \alpha\), where \(\bar{p} =
   \NN^*(\NN_{\bar{\omega},\bar{c}} - y)\) is the associated dual variable.
 \end{itemize}
Then, \(\bar{\mu}\) is a local minimum of \(J\) (i.e., a local solution of~\eqref{eq:Phi_problem}).
\end{theorem}
\begin{proof}
First, since \(\bar{c}\) is a local minimum of \(J_{\bar{\omega}}\), there exists
an \(\epsilon>0\), such that \(J_{\bar{\omega}}(c) \geq J_{\bar{\omega}}(\bar{c})\) for
all \(c \in \R^N\) with \(\norm{c-\bar c}_{\ell_1} \leq \epsilon\). 
Due to Remark~\ref{rem:finite} \fixed{and~\ref{cond:phi_2}},
\(\abs{\bar{p}(\bar{\omega}_n)} = \alpha\phi'(\fixed{\abs{\bar{c}_n}}) < \alpha\phi'(0) =
\alpha\),  and thus there exists a \(\delta > 0\), such that
\(\sup_{\omega\in \Omega}\abs{\bar{p}(\omega)} \leq (1-\delta)\alpha\).
Without restriction, assume in the following that
\(\epsilon \leq \delta\alpha/(\|\NN\|^2+(\alpha\gamma/2))\), where \(\gamma\) is from~\ref{cond:phi_1}.
To verify local optimality of \(\bar{\mu}\), let \(\mu \in M(\Omega)\) be arbitrary with
\(\norm{\mu - \bar{\mu}}_{M(\Omega)} \leq \epsilon\).
By \fixed{decomposing \(\mu\) into atomic and non-atomic part}, we can
write
\[
\mu = \mu_0 + \widetilde\mu,
\quad\text{with }
\mu_0 = \sum_{n=1}^N c_n \delta_{\bar{\omega}_n}
\quad\text{and }
\widetilde{\mu} = \sum_{n} \widetilde{c}_n \delta_{\widetilde{\omega}_n} +
\mu_{\text{cont}},
\]
where \fixed{\(\{\,\widetilde{\omega}_n\,\} = \atom \mu \setminus \atom\bar{\mu}\)},
\({c}_n = {\mu}(\{\,{\fixed{\bar{\omega}}}_n\,\})\), \(\widetilde{c}_n =\widetilde{\mu}(\{\,\widetilde{\omega}_n\,\})\),
and \(\mu_{\text{cont}}\) is the continuous part of \(\mu\).
Therefore, it follows \(\NN\mu = \NN\mu_0 + \NN\widetilde{\mu}\).
Moreover, \(\norm{c-\bar{c}}_{\ell_1} + \norm{\widetilde{c}}_{\ell_1} + \norm{\mu_{\text{cont}}}_{M(\Omega)}
= \norm{\mu - \bar{\mu}}_{M(\Omega)}\leq \epsilon\).
By the quadratic form of the loss, we obtain
\begin{align*}
\frac{1}{2} \norm{\NN\mu - y}^2_{\Hi}
&= \frac{1}{2} \norm{\NN\mu_0 - y}^2_{\Hi}
  + \inner{\NN\mu_0 - y, \NN\widetilde{\mu}}_{\Hi}
  + \frac{1}{2}\norm{\NN\widetilde{\mu}}^2_{\Hi} \\
&\geq
 \frac{1}{2}\norm{\NN_{\bar{\omega}, c} - y}^2_{\Hi}
+ \inner{\NN[\mu_0-\bar{\mu}] , \NN\widetilde{\mu}}_{\Hi} 
+ \inner{\NN\bar{\mu} -  y, \NN\widetilde{\mu}}_{\Hi} \\
&\geq
 \frac{1}{2}\norm{\NN_{\bar{\omega}, c} - y}^2_{\Hi}
- \norm{\NN}^2\epsilon \fixed{\norm{\widetilde{\mu}}_{M(\Omega)}}
+ \pair{\bar{p},\widetilde{\mu}}, 
\end{align*}
using that
\[
  |\inner{\NN[\mu_0-\bar{\mu}] , \NN\widetilde{\mu}}_{\Hi}|\leq\|\NN\|^2
  \|c-\bar{c}\|_{\ell_1}\norm{\fixed{\widetilde{\mu}}}_{M(\Omega)}
  \leq \fixed{ \|\NN\|^2\epsilon \norm{\widetilde{\mu}}_{M(\Omega)}}.
\]
Moreover, for the penalty it holds
\begin{align*}
\Phi(\mu)
= \Phi(\mu_0) + \Phi(\widetilde{\mu})
= \sum_{n=1}^N\phi(\abs{c_n})
 + \sum_{n}\phi(\abs{\widetilde{c}_n})
 + \int_{\Omega} \de\abs{\mu_{\text{cont}}}.
\end{align*}
Combining this, we obtain
\begin{align*}
J(\mu)
&\geq J_{\bar{\omega}}(c) 
- \fixed{ \|\NN\|^2\epsilon \norm{\widetilde{\mu}}_{M(\Omega)}}
+ \int_\Omega \left[\alpha - \abs{\bar{p}}\right] \de\abs{\mu_{\text{cont}}}
+ \sum_{n} \left[\alpha \phi(\abs{\widetilde{c}_n}) -
  \abs{\bar{p}(\widetilde{\omega}_n)}\abs{\widetilde{c}_n}\right].
\end{align*}
By the optimality of \(\bar{c}\), it follows that \(J_{\bar{\omega}}(c) \geq
J_{\bar{\omega}}(\bar{c}) = J(\bar{\mu})\). For the \fixed{third} term, we use that
\[
  \int_\Omega \left[\alpha - \abs{\bar{p}}\right] \de\abs{\mu_{\text{cont}}}
  \geq \delta\alpha \fixed{\int_\Omega\de\abs{\mu_{\text{cont}}}}.
\]
\fixed{For the fourth term we use~\ref{cond:phi_1} for 
\(\phi(\abs{\widetilde{c}_n}) = \phi'(0)\abs{\widetilde{c}_n} + \int_0^{\abs{\widetilde{c}_n}}\phi'(z)-\phi'(0)\de z
\geq \abs{\widetilde{c}_n} - (\gamma/2) \abs{\widetilde{c}_n}^2\) to obtain}
\begin{align*}
\sum_{n} \left[\alpha \phi(\abs{\widetilde{c}_n}) - \abs{\bar{p}(\widetilde{\omega}_n)}\abs{\widetilde{c}_n}\right]
  &\geq \fixed{\sum_{n} \left[\left(\alpha  - \abs{\bar{p}(\widetilde{\omega}_n)}\right)\abs{\widetilde{c}_n}
    - (\alpha\gamma/2) \abs{\widetilde{c}_n}^2\right]} \\
  &\geq \fixed{\delta\alpha \norm{\widetilde{c}}_{\ell_1}
    - (\alpha\gamma/2) \norm{\widetilde{c}}_{\ell_2}^2
    \geq \delta\alpha \norm{\widetilde{c}}_{\ell_1}
    - (\alpha\gamma/2) \norm{\widetilde{c}}_{\ell_1}^2}.
\end{align*}
Combining these estimates, we consequently obtain
\begin{align*}
  J(\mu)
  &\geq J(\bar{\mu}) + \fixed{\delta \alpha \left(\int_\Omega\de\abs{\mu_{\text{cont}}} + \norm{\widetilde{c}}_{\ell_1}\right) -
  (\alpha\gamma/2) \norm{\widetilde{c}}_{\ell_1}^2
    - \|\NN\|^2\epsilon \norm{\widetilde{\mu}}_{M(\Omega)}} \\
  &\geq J(\bar{\mu}) + \fixed{\left[\delta \alpha  -
  \left((\alpha\gamma/2)
    + \|\NN\|^2\right)\epsilon \right]\norm{\widetilde{\mu}}_{M(\Omega)}}\\
  &\geq J(\bar{\mu}).
  \qedhere
\end{align*}
\end{proof}
We can interpret the conditions from the previous result in terms of the original
problem~\eqref{eq:phi_problem}. Condition~i) simply requires the outer
weights \(c\) in~\eqref{eq:phi_problem} to be chosen as a local minimum. Condition~ii)
can be read as follows: Adding any number of additional nodes \(\omega\)
with small outer weights to \(\NN\) will increase the training
objective of~\eqref{eq:phi_problem}.

The next theorem proves that any local minimizer is not only \fixed{finitely supported},
guaranteed by Theorem~\ref{thm:local_finiteness}, but also fits the
data properly. To state the theorem we need to introduce the following
notation; cf.~\cite{bach2017breaking, 2019arXiv191001635O}.  Let $\mathcal{W}(D)$ be the space of functions
$f$ on $D$ that satisfy \(f(x)=[\NN\mu](x)\) for every
$x\in D$ and some $\mu\in M(\Omega)$.
\fixed{A characterization of this space for the ReLU activation function will be provided
below in Section~\ref{sec:exact_rep}.}
For $f\in\mathcal{W}(D)$, denote
\begin{equation}\label{eq:c_f}
  \|f\|_{\mathcal{W}(D)}=\min_{\mu\in M(\Omega)}\norm{\mu}_{M(\Omega)}
  \quad\text{subject to } f(x)=[\NN\mu](x)\quad \text{for every } x\in D.
\end{equation}
Note that, for any $f\in \mathcal{W}(D)$, there exists a minimizer \(\mu_f\)
with \(f = \NN \mu_f\) such that \( \|f\|_{\mathcal{W}(D)} = \norm{\mu_f}_{M(\Omega)}\).
The existence of the optimal measure in~\eqref{eq:c_f} follows from
the direct method of variational calculus by showing that a minimizing sequence has a
\fixed{convergent subsequence in the weak-\(*\) sense (cf.\ Appendix~\ref{sec:phi_reg})}.

\begin{theorem}\label{thm:fid_est}
If $\phi$ satisfies conditions~\ref{cond:phi_1}, and
$f\in\mathcal{W}(D)$ then, for any local solution \fixed{\(\bar{\mu}\)} of~\eqref{eq:Phi_problem},
\[
\norm{\NN\bar{\mu} - f}_{\Hi}^2
 \leq  2\, \alpha\, \norm{f}_{\mathcal{W}(D)} + \norm{y - f}_{\Hi}^2.
\]
\end{theorem}
\begin{proof}
Let $\mu_f$ be such that $f(x)=[\NN\mu_f](x)$ for all $x\in D$ and $\norm{\mu_f}_{M(\Omega)} =  \|f\|_{\mathcal{W}(D)}$. 
Then, for any local minimizer $\bar{\mu}$ of~\eqref{eq:Phi_problem}, we have
\begin{align*}
 \norm{\NN\bar{\mu} - f}_{\Hi}^2&=
 \inner{\NN\bar{\mu} - y, \NN(\bar{\mu} - \mu_f)}_{\Hi}
 +\inner{y - f, \NN\bar{\mu} - f}_{\Hi}
 \\
 &\leq\pair{\NN^*(\NN\bar{\mu} - y), \bar\mu - \mu_f}
 + \frac{1}{2}\norm{y-f}^2_{\Hi} + \frac{1}{2} \norm{\NN\bar\mu - f}_{\Hi}^2,
\end{align*}
using Young's inequality. With \(\bar{p} = \NN^*(\NN\bar{\mu} - y)\) and bringing the
last term to the left-hand side, we arrive at
\begin{align*}
  \frac{1}{2}\norm{\NN\bar{\mu} - f}_{\Hi}^2
  \leq  \pair{\bar{p}, \bar\mu} - \pair{\bar{p}, \mu_f}
 + \frac{1}{2}\norm{y-f}^2_{\Hi}.
\end{align*}
Now, we can estimate the first term by zero, due to
\begin{align*}
 \pair{\bar{p},\bar{\mu}}
= - \alpha \int_\Omega
  \phi'(\abs{\bar{\mu}(\{\,\omega\,\})})\sign(\bar{\mu})(\omega)\de\bar{\mu}(\omega) 
= - \alpha \int_\Omega  \phi'(\abs{\bar{\mu}(\{\,\omega\,\})})\de\abs{\bar{\mu}}(\omega)
\leq 0 ,
\end{align*}
using the optimality conditions for \(\bar{\mu}\)
from Theorem~\ref{thm:opt_cond}. Finally, the second term is estimated as
\(-\pair{\bar{p}, \mu_f} \leq \alpha \norm{\mu_f}_{M(\Omega)}\) using
\(\norm{\bar{p}}_{C(\Omega)}\leq \alpha\), resulting in the desired estimate.
\end{proof}

The condition $f\in \mathcal{W}(D)$ may appear restrictive, but it turns out that
a large class of functions is included.  In particular, when
$\sigma$ is given by the ReLU function, all sufficiently smooth functions are contained
in $\mathcal{W}(D)$. Moreover, if additionally $D = \R^d$, the
quantity $\|f\|_{\mathcal{W}(\R^d)}$ can be explicitly computed. We will discuss this topic further in
Section~\ref{sec:exact_rep}.  

Finally, we give an additional upper bound in the case of finite data.
The representer theorem for the $\ell_1$ minimization (see, e.g.,
\cite[Section~2.2]{bach2017breaking}) claims that the problem~\eqref{eq:L1_problem} has a global
solution with at most $K$ nodes. A similar but stronger version of the representer
theorem holds for the problem~\eqref{eq:Phi_problem}, which is proved in Appendix~\ref{sec:representer}.
\begin{theorem}\label{thm:representer_phi}
  If $\phi$ satisfies conditions~\ref{cond:phi_1} and~\ref{cond:phi_2}, and $\supp(\nu) =
  \{\,x_k\,\}_{k=1,\ldots,K}$ is a finite set with $K$ points then  any
  local solution of \eqref{eq:Phi_problem} has at most $K$ nodes.
\end{theorem}

\section{Exact representation with integral ReLU neural networks}\label{sec:exact_rep}
We return to the space $\mathcal{W}(D)$, more specifically to the case
when the activation function is the ReLU function
\fixed{\(\sigma(x;\omega) = \max\{\,0,\,a\cdot x+ b\}\) and $(a,b) = \omega \in \Omega = \Sd$}.
\fixed{Here, it can be easily seen that any $f \in \mathcal{W}(D)$ is Lipschitz continuous
  on $D$.}
First, we consider $D=\R^d$,
\fixed{where \(\nu\) is any probability measure supported on the whole \(\R^d\)}.
We can now characterize the kernel of \(\NN\): It is given
by $M_{-}(\Sd)$, the set of all odd measures $\mu$,
i.e.\ $\de\mu(-a,-b)=-\de\mu(a,b)$, that satisfy the conditions
\[
a_\mu := \int_{\Sd} a\de\mu(a,b)=0\;
\text{ and }\; b_{\mu} := \int_{\Sd} b\de\mu(a,b)=0.
\]
It can be seen that $\NN\mu\equiv 0$, if and only if $\mu\in M_{-}(\Sd)$.
\fixed{
Indeed, by noting that 
\[
  \max\{0;a\cdot x + b\} = \frac{1}{2}(a\cdot x + b - \abs{a\cdot x + b}),
\]
for any measure \(\mu\) it holds
  \begin{align}\label{eq:oddmeasureformula}
  \begin{split}
    \NN\mu 
    &= \frac{1}{2} \left(\int_{\Sd} \left[a\cdot x + b\right] \de \mu+ \int_{\Sd} \abs{a\cdot x + b} \de \mu\right) \\
    &= \frac{1}{2} \left(a_{\mu} \cdot x + b_{\mu} + \int_{\Sd} \abs{a\cdot x + b} \de \mu\right).
    \end{split}
  \end{align}
  Note, that the first two terms only depend on the odd part of \(\mu\), and the last term
  only on the even part.
  Thus, if $\mu\in M_{-}(\Sd)$, then $\NN\mu=0$. The converse follows from
  \cite[Lemma~10]{2019arXiv191001635O} which, by a change of variables, claims that for
  every $f=\NN\mu$
  there exists a unique linear function \(a\cdot x + b\) and a unique even measure
  \(\mu_{f}\) with \(f = a\cdot x + b + \NN \mu_f\). Thus, for $\NN\mu=0$ the unique even measure will be the
  zero measure and \((a,b) = 0\). Taking into account~\eqref{eq:oddmeasureformula},
  we must therefore have that \(\mu\) is an odd measure with $a_\mu=0$ and $b_\mu=0$.
}
Next, we derive an explicit formula for the \fixed{even} measure \(\mu_f\) that represents exactly a
given smooth function \(f\). 
\begin{theorem}\label{thm:rep_smooth}
For a compactly supported function $f\in C_{c}^{d+1}(\R^d)$, we define the coefficient function
\[
    c_f(a,b) = \left\{
    \begin{array}{ll}
        \displaystyle{
        \frac{(-1)^{\nicefrac{(d+1)}{2}}}{2(2\pi)^{d-1}}  \frac{1}{\|a\|^{d+2}} \,
        \frac{\partial^{d+1}}{\partial b^{d+1}} \mathcal{R}[f](a,b)  }
        &\quad \text{if } d \text{ is odd};
        \\
        \displaystyle{
        \frac{(-1)^{\nicefrac{d}{2}}}{2(2\pi)^{d-1}}  \frac{1}{\|a\|^{d+2}} \,
    	\frac{\partial^{d+1}}{\partial b^{d+1}} \mathcal{H}\big[ \mathcal{R}[f](a,\fixed{\cdot}) \big](b) \, }
        &\quad \text{if } d \text{ is even}.
    \end{array}
    \right.
  \]
\fixed{
Here, $\mathcal{R}[f]$ is the Radon transform of $f$  and
$\mathcal{H}[g]$ is the Hilbert transform of a function $g\colon \R \to \R$.%
}
Then $f(x)=[\NN\mu_f](x)$ where $\de\mu_f(a,b) =c_f(a,b) \de(a,b)$ is the measure with
density $c_f(a,b)$ with respect to the \fixed{\(d\)-dimensional Hausdorff measure} \(\de(a,b)\) on $\Sd$.
Moreover,  $f(x)=[\NN\mu](x)$ for $\mu\in M(\fixed{\Sd})$  if and only if 
\( \mu = \mu_f + \mu_{-}\), where $\mu_{-}\in M_{-}(\Sd)$,
and 
\[
\|f\|_{\mathcal{W}(\R^d)} = \norm{\mu_f}_{M(\Omega)} = \norm{c_f}_{L^1(\Omega)},
\]
where $\|f\|_{\mathcal{W}(\R^d)}$ is defined in~\eqref{eq:c_f}.
\end{theorem}
\begin{proof}
\fixed{The proof of Theorem~\ref{thm:rep_smooth}  follows by combining \cite[Theorem~1]{2019arXiv191002743D} and
\cite[Lemma~10]{2019arXiv191001635O}.}
\end{proof}

Concerning approximation on a bounded domain, for  $f\in C^{d+1}(D)$, we have
$C^{d+1}(D)\subset \mathcal{W}(D)$, and Theorem~\ref{thm:rep_smooth} provides an upper bound
for $\|f\|_{\mathcal{W}(\R^d)}$. Here we say that $f\in C^{d+1}(D)$ if there exists an extension $F\in C^{d+1}_c(\R^d)$ such that $F\rvert_{D} = f$.
\begin{corollary} \label{cor:formula_wRD}
If $f\in C^{d+1}(D)$,
then $f \in \mathcal{W}(D)$, and for any extension $F \in C_c(\R^d)$ with
$F\rvert_{D} = f$, we have
\[ \|f\|_{\mathcal{W}(D)} \leq  \|F\|_{\mathcal{W}(\R^d)} = \norm{\mu_F}_{M(\Sd)}.\]
\end{corollary}
\begin{proof}
  It easily follows from the observation that
  \[
    \{\mu\in M(\Sd):\; f(x)=[\NN\mu](x),\; \forall x\in \fixed{\R^d}\}
    \subset  \{\mu\in M(\Sd):\; F(x)=[\NN\mu](x),\; \forall x\in \fixed{D}\}.
    \qedhere
  \]
\end{proof}

Additionally, in the setting where $D$ is bounded,
we point out that the \(\mathcal{W}\) norm can be estimated by what is know as
the Barron constant in the literature (going back to~\cite{barron1993universal}
and~\cite{breiman1993hinging}). In
particular, it can be shown that, for any continuous function
$f\colon \R^d \to \R$ with Fourier transform \(\widehat{f}\) and 
\[
C_f = \int_{\R^d} \norm{\omega}^2 |\widehat{f}(\omega)| \de \omega < \infty,
\]
the restriction of $f$ to the bounded set \(D\) is in $\mathcal{W}(D)$ with norm bounded by a constant
factor of \(C_f\). We refer to the introduction of~\cite{klusowski2016risk}
(cf.\ also \cite{ma2019barron}), where we note that the
variation of the function \(f\) introduced there is equivalent to the
definition~\eqref{eq:c_f} given above due to density of \fixed{finite linear combinations of} Dirac-delta functions in the space
of measures \fixed{(in the weak-\(*\) sense)} and the equivalence in \(\R^{d+1}\) of the \(1\)-norm to the Euclidean norm.
This characterizes a different subset of \(\mathcal{W}(D)\) (different from \(C^{d+1}\)).

We conclude this section with the following observation: A smooth function is
\fixed{optimally} represented by a smooth density \(c(a,b)\) as opposed to a sparse discrete
measure~\eqref{eq:measure_discrete}. Thus, we can not expect the solution
of~\eqref{eq:L1_problem} (which can be considered an approximate version
of~\eqref{eq:c_f}) to provide such a discrete measure in the infinite data case.
Indeed, in the numerical experiments of Section~\ref{sec:numerics}, we will observe that
the optimal solution of~\eqref{eq:L1_problem} will tend to approximate the continous
density rather than a maximally discrete one, in contrast to the nonconvex~\eqref{eq:Phi_problem}.



\section{The optimization algorithm}\label{sec:gen_conj_grad}

To numerically solve the problem \eqref{eq:phi_problem}, we
deploy the following algorithm which consists of three phases that
are executed consecutively. The first phase adds new neurons in a greedy fashion, the
second optimizes the weights, and redundant neurons are pruned in phase three as detailed below.
The proposed method can be considered an accelerated version of  the conditional gradient
method \cite{bach2017breaking,2019arXiv190409218P} and is also closely related to the gradient
boosting \cite{friedman2001greedy} \fixed{and the CoSaMP \cite{needell2009cosamp} algorithms}.

\begin{algorithm}[h]
  \caption{Iterative node insertion and optimization}\label{alg:our_algorithm}
  \label{alg:practical}
 \begin{algorithmic}[1]
   \STATE Initial network $\omega^{(0)},c^{(0)}$ of width \(N(0)\)
   \WHILE{$t < T$}
   \STATE Sample \(N_{\text{trial}}\) random nodes \(\omega \in \Omega\).
   \STATE Optimize (in parallel) the function~\eqref{eq:max_dual} starting from the
   initial nodes.
   \STATE Select nodes with \(\abs{p_t(\omega)} > \alpha\) and add
   them to the network (of width \(N(t+1/2)\)).
   \STATE Perform local training based on~\eqref{eq:local_training}.
   \STATE Remove nodes with outer weight zero (resulting in width \(N(t+1)\)).
   \STATE $t=t+1$.
   \ENDWHILE
 \end{algorithmic} 
\end{algorithm}

\paragraph{Initialization:}
Let $\omega^{(t)} = [\omega^{(t)}_1,\ldots,\omega^{(t)}_{N(t)}]$,
$c^{(t)}=[c_1^{(t)},\ldots,c^{(t)}_{N(t)}]$ be the lists of network inner and outer weights
in the $t$-th iteration and $N(t)$ be the corresponding number of neurons.
Network initialization is arbitrary:
one can start from any network, including the empty network, then add
and extract neurons to derive an optimal network.

\paragraph{Phase 1} 
To determine new points to insert, in the greedy insertion step, we compute the
nodes \(\omega\in\Omega\) for which the correlation of $\sigma(\cdot;\omega)$ with the residual
\(g_t = \NN_{\omega^{(t)},c^{(t)}} - y\) is largest. Thus, we maximize 
\begin{equation}
  \label{eq:max_dual}
\Omega \ni \omega \mapsto
  \abs{p_t(\omega)}
  =
  \abs*{\frac{1}{K}\sum_{k=1}^K \sigma(x_k; \omega) g_t(x_k)}
\end{equation}
where we have assumed that \(K\) is finite.
Note that $p_t(\omega)$ is exactly the dual variable as defined in
Section~\ref{subsec:local_sol}.
Finding a global maximum of~\eqref{eq:max_dual} (in a high dimensional space) is a challenging problem, and its
reliable determination up to a guaranteed
tolerance for the specific problem here is subject of ongoing research; cf.\ \cite{bach2017breaking}.
As an ersatz, we use the following heuristic which is
commonly employed in practice: we test all local maxima of~\eqref{eq:max_dual} that are
found by a gradient maximization, initialized at \(N_{\text{trial}}\) random
points on $\Omega$.  This corresponds to solving \(N_{\text{trial}}\)  simple
\fixed{constrained} optimization problems (in parallel); cf.,
e.g.,~\cite{boyd2017alternating}.
\fixed{Moreover, the constraint \((a_n,b_n) = \omega_n \in \Sd\) can be removed by parametrization of the sphere,
using, e.g., a stereographic projection; see Appendix~\ref{sec:sup_gen_conj_grad}.}
Of these points, we insert all that violate the constraint \(|p_t(\omega)| \leq \alpha\) (after
removing possible duplicates).


\paragraph{Phase 2}
Let $N(t+1/2)\leq N(t)+N_{\text{trial}}$ denote the number of neurons in the resulting
network. Next, we compute an approximate local solution to
the objective function of~\eqref{eq:l1_problem}, for \(N = N(t+1/2)\), given by
\begin{equation}
  \label{eq:local_training}
  (\Omega\times\R)^{N(t+1/2)} \ni (\omega,c) \mapsto
  l\left(\NN_{\omega, c};y\right) + \alpha\sum_{n=1}^{N(t+1/2)}\phi(\abs{c_n}).
\end{equation}
\fixed{Here, it is} sufficient to optimize with respect to the outer weights only
\fixed{according to Theorem~\ref{thm:nec_suff_local_sol}}.
The resulting nonsmooth optimization problem can be solved by
standard training methods based on (proximal) gradient descent, using the old values as initialization for the weights from the
previous iteration.
\fixed{In practice, we employ more efficient second order semismooth Newton methods;
  see Appendix~\ref{sec:sup_gen_conj_grad}.}
Alternatively, we can optimize in terms of all weights, \fixed{using an appropriate
  nonsmooth optimization algorithm.}

\paragraph{Phase 3}
Finally, we note that the objective~\eqref{eq:local_training} contains a sparsity
promoting term for the outer weights, and we can expect several of the outer weights to be
zero after having solved the problem up to a specified tolerance. Thus, we can eliminate
these outer weights (together with their corresponding inner weights) from the network
without changing the underlying function. This results in a new network of final width \(N(t+1)\).

The resulting procedure is outlined in Algorithm~\ref{alg:practical}.
In Appendix~\ref{sec:sup_gen_conj_grad}, we provide a more
detailed explanation of the implementation of single steps of the method.
\fixed{The code for the numerical experiments is provided at~\cite{NonConvexNNcode}}.

\section{Numerical examples}
\label{sec:numerics}
In this section, we supply numerical  examples in one and two
dimensions. We demonstrate the difference between \(\ell_1\) and
\(\phi\) regularization and the effect of the parameter $\gamma$
on the optimal number of neurons. Here, we take \(\phi\) as
\begin{equation}\label{eq:func_phi_gamma}
\phi_\gamma(z) = \frac{1}{2} \left( z + \phi_{\log,2\gamma}(z) \right),
\end{equation}
which fulfills~\ref{cond:phi_1} and~\ref{cond:phi_2}
with $\widehat{\gamma} = \gamma/\fixed{4}$ and \(\widehat{z} = 1/(\fixed{2}\gamma)\), where
$\phi_{\log,2\gamma}(z)$ is defined in \eqref{form:phi_log}.

\subsection{One dimensional  example}
\label{sec:numerics_1}
In Figure \ref{fig:1d_comparision}, we consider the function
\begin{equation*}
f(x) = \exp\left(-x^2/2\right)\abs*{\sin\left(7\sqrt{1+x^2}\right)}
\end{equation*}
on the interval $D = [-1,1]$. The data set consists of $1000$ uniformly arranged points
\(x_k\) and \(y_k = f(x_k)\) (without noise).
We fix $\alpha = \num{e-05}$ and the number of iterations of Algorithm~\ref{alg:practical} is $\fixed{T}=15$.
During each iteration, up to $N_{\text{trial}} = 50$
nodes can be added to the network.
In Figure~\ref{fig:1d_comparision}, we plot the results for~\eqref{eq:l1_problem}
and~\eqref{eq:phi_problem} with \(\gamma = 1\). 
\begin{figure}[h]
\centering
	\begin{subfigure}{.49\linewidth}
		\centering
                \includegraphics[width=0.9\textwidth]{paper_test_1/l1_forward.tikz}
                \includegraphics[width=0.9\textwidth]{paper_test_1/l1_adjoint.tikz}
	  \caption{Global \(\ell_1\)-solution ($\gamma=0$).}
        \end{subfigure}
	\begin{subfigure}{.49\linewidth}
		\centering
                \includegraphics[width=0.9\textwidth]{paper_test_1/phi_forward.tikz}
                \includegraphics[width=0.9\textwidth]{paper_test_1/phi_adjoint.tikz}
	\caption{Local \(\phi_\gamma\)-regularized solution for $\gamma =1$.}
	\end{subfigure}
	\caption{Comparison between solutions of~\eqref{eq:l1_problem} and
          \eqref{eq:phi_problem}, for \(\phi_\gamma(z)\) in~\eqref{eq:func_phi_gamma} in one dimension.
          \textit{Top:} True function samples (dashed blue) and neural network
          approximation (black, orange crosses at the knot points \(x_n= -b_n/a_n\)).
          \textit{Middle:} Optimal sparse measure in angular coordinates.
          \textit{Bottom:} Optimal dual variable in angular coordinates.
        }\label{fig:1d_comparision}
\end{figure}
We observe that the \(\ell_1\) optimal network, while providing a sparse approximation of
the function in some areas, contains dense ``clusters'' of inner weights on the sphere,
corresponding to a large cluster of knot points in the corresponding linear spline. These
are located in areas of the sphere where the dual variable assumes the extreme values
\(\pm\alpha\). For the \(\phi\) minimal network the knot points
appear more sparsely spaced and the dual variable is reduced below the extreme values.
To further investigate the influence of the hyperparameter $\gamma$, we also solve~\eqref{eq:phi_problem} for
\(\gamma \in \{\,\num{e-3},\,\num{e-2},\,\num{e-1}\,\}\), and give their optimal
number of neurons and approximation quality in Table~\ref{table:gamma_com_1d}. We observe that
increasing \(\gamma\) leads to an increased reduction in the number of neurons while keeping
the approximation quality essentially constant. We note that \(\gamma=0\) is the
\(\ell_1\) solution.
\begin{table}[h]
\centering
\begin{subfigure}{.49\linewidth}\centering
\small
\begin{tabular}{lrr}
\toprule
$\gamma$       & $N$  & $\norm{\NN_{\omega,c} - f}$   \\ \midrule
\num{0}    & \num{342} & \num{0.0045} \\
\num{e-03} & \num{71}  & \num{0.0044} \\
\num{e-02} & \num{35}  & \num{0.0043} \\
\num{e-01} & \num{20}  & \num{0.0039} \\
\num{1}    & \num{15}  & \num{0.0038} \\
\bottomrule
\end{tabular}
\caption{Example from Section~\ref{sec:numerics_1}.
}\label{table:gamma_com_1d}
\end{subfigure}
\begin{subfigure}{.49\linewidth}\centering
\small
\begin{tabular}{lrr}
\toprule
$\gamma$        & $N$     & $\norm{\NN_{\omega,c} - f}$ \\ 
\midrule
\num{0}        & {105} & \num{1.9e-04} \\
\num{e-02}     & {53}  & \num{1.9e-04} \\
\num{e-01}     & {29}  & \num{3.3e-04} \\
\num{1}        & {16}  & \num{8.9e-04} \\
\num{10}       & {14}  & \num{1.0e-3} \\
\bottomrule
\end{tabular}
\caption{Example from Section~\ref{sec:numerics_2}.
}\label{table:gamma_com_2d}
\end{subfigure}
\caption{Optimal $N$ and fidelity for the solutions of \eqref{eq:phi_problem}
  with~\eqref{eq:func_phi_gamma} and different $\gamma$.
}\label{table:gamma_com}
\end{table}

\subsection{Two-dimensional  example}
\label{sec:numerics_2}
For the two dimensional experiment, we uniformly sample
\begin{equation}
\label{eq:example_2d}
f(x_1,x_2) = \exp\left(-\frac{x_1^2+x_2^2}{2}\right)\,\cos(10\,x_1x_2).
\end{equation}
on a $31\times 31$ grid in $D=[-1,1]^2$ and solve \eqref{eq:l1_problem} and
\eqref{eq:phi_problem} with \(\phi_\gamma\) from \eqref{eq:func_phi_gamma} and $\alpha=\num{e-5}$
and $\gamma=5$. For both methods, the algorithm is iterated \fixed{\(T=10\)} times, with
up to \fixed{\(N_{\text{trial}} = 50\)} nodes added in each
iteration.
We give the results in Figure~\ref{fig:2d_plot}. As before, we observe a reduction in the
number of nodes for the nonconvex regularization with essentially the same approximation
quality. In contrast to the convex regularized solution, the nonconvex regularization is not
affected by the clustering of nodes in a small area of the sphere.
\begin{figure}[h!]
	\begin{subfigure}[b]{.3\linewidth}
	\centering
        \includegraphics[width=0.99\textwidth]{paper_test_2/l1_measure.tikz}
        \caption{Optimal measure \eqref{eq:l1_problem}:\\
          \(\gamma=0\), \(N = 197\).}
	\end{subfigure}
	\begin{subfigure}[b]{.3\linewidth}
          \centering
          \includegraphics[width=0.99\textwidth]{paper_test_2/phi_measure.tikz}
          \caption{Optimal measure \eqref{eq:phi_problem}:\\
            $\gamma=5$, \(N = 71\).}
	\end{subfigure}
	\begin{subfigure}[b]{.35\linewidth}
          \centering
         \includegraphics[width=0.99\textwidth]{paper_test_2/l1_function.tikz}
         \caption{
           Optimal network and data (black circles) for~\eqref{eq:l1_problem}:
           \(\norm{\NN_{\omega,c} - f} = \num{9.4e-3}\). \\
           The plot for~\eqref{eq:phi_problem} with $\gamma=5$ is indistinguishable:
           \(\norm{\NN_{\omega,c} - f} = \num{8.7e-3}\).}
        \end{subfigure}
	\caption{Comparison of  \eqref{eq:l1_problem} and
          \eqref{eq:phi_problem} with~\eqref{eq:func_phi_gamma} for~\eqref{eq:example_2d}.
	\textit{Left:} Location of optimal inner weights (color indicates sign, size
        outer weight magnitude) under stereographic projection of~\(\mathbb{S}^2\) from the pole \((a,b)=(0,-1)\).
	\textit{Right:} Data points and optimal neural network.}
      \label{fig:2d_plot}
\end{figure}

To \fixed{illustrate the effect of increased \(\gamma\) on removal of ``clusters''} in
more detail we \fixed{consider a function with analytically known representing measure \(\mu_f\)
given by}
\begin{equation}
\label{eq:example_abs}
f(x) = \norm{x - \hat{x}} = \sqrt{(x_1-\hat{x}_1)^2+(x_2-\hat{x}_2)^2},
\quad\text{where } \hat{x} = (0.1,0.1),
\end{equation}
which is uniformly sampled on a $21\times 21$ grid in $D = [-1,1]^2$.
The results are provided in Table~\ref{table:gamma_com_2d}.
Concerning the number of neurons, we observe a steady reduction
\fixed{from \(N=105\) to \(N=14\) with increasing \(\gamma\)}. In this example the approximation quality is slightly
reduced; however, it is still consistently small following the estimate
\(\norm{\NN_{\omega,c} - f}^2 \leq \alpha\norm{f}_{\mathcal{W}(D)}\) in Theorem~\ref{thm:fid_est}.

For this example, we can give the integral representation exactly (cf.\ Theorem~\ref{thm:rep_smooth}).
One can show that \(f = \NN\mu_f\) for a measure \(\mu_f\) supported on
the great circle \(S_{\hat{x}} = \{\,(a,b) \in \mathbb{S}^2 \;|\; a \cdot \hat{x} = b\,\}\):
\begin{equation*}
f(x)
= \int_{S_{\hat{x}}} \max\{\,a \cdot x + b,\, 0\,\} \frac{1}{2\norm{a}} \de S(a,b)
= \frac{1}{2}\left[ \int_{S_0} \max\{\,a \cdot y,\,0\,\} \de S(a,b) \right]_{y = x - \hat{x}}.
\end{equation*}
Here, \(\de S\) is the one-dimensional line integral on
\(S_{\hat{x}}\) (resp.\ \(S_0\)).
Therefore, \(\mu_f\) is given by \(\de\mu_f = 1/(2\norm{a}) \de
S(a,b)\rvert_{S_{\hat{x}}}\).
In Figure~\ref{fig:2d_duals} we investigate more closely the properties of the solutions.
\begin{figure}[h!]
	\begin{subfigure}{.3\linewidth}\centering
          \includegraphics[width=0.99\textwidth]{paper_test_3/l1_measure.tikz}
          \caption{$\gamma = \num{0}$, $N = 105$.}
	\end{subfigure}
	\begin{subfigure}{.3\linewidth}\centering
          \includegraphics[width=0.99\textwidth]{paper_test_3/phi_measure.tikz}
          \caption{$\gamma = \num{10}$, $N = 14$.}
	\end{subfigure}
	\begin{subfigure}{.38\linewidth}\centering
          \includegraphics[width=0.9\textwidth]{paper_test_3/phi_adjoint.tikz}
          \caption{Dual variable for $\gamma = \num{10}$.}
	\end{subfigure}
	\caption{
          Comparison of  \eqref{eq:l1_problem} and
          \eqref{eq:phi_problem} with~\eqref{eq:func_phi_gamma} for~\eqref{eq:example_abs}.
          \textit{Left:} Location of inner weights (dot color indicates sign, dot size
          outer weight magnitude) under stereographic projection of~\(\mathbb{S}^2\) from the pole \((a,b)=(0,-1)\).
          \textit{Right:} Dual variable and nodes under stereographic projection.}\label{fig:2d_duals}
\end{figure}
We observe that the nodes of the solution of~\eqref{eq:l1_problem} densely
cluster on the great circle \(S_{\hat{x}}\), up to gaps arising due to
the symmetry of the nodes with respect to \((a,b) \sim (-a,-b)\)
(cf.\ Theorem~\ref{thm:rep_smooth}).
This is remarkable since the support of the regularized solution -- representing a
compromise between a fit of the data and a small regularization term -- is essentially
the same as the exact measure \(\mu_f\) that achieves the perfect fit.
In contrast, for large \(\gamma\),
the points are sparsely placed on the circle and regularly spaced
up to the symmetry \((a,b) \sim (-a,-b)\). By reducing the dual variable
below the bound \(\abs{\bar{p}} \leq \alpha\) in the existing nodes, no additional
point with a small weight can be inserted without increasing the regularized objective.

\subsection{Three-dimensional example}
\fixed{
Here, we consider for \(d=3\) the function
\begin{equation}
  \label{eq:f_3d}
  f(x_1,x_2,x_3) = (x_2^2 + x_3^2)/(x_1 + 1.05) + 2(x_1 + 1.05)^4
\end{equation}
on a random set of $1000$ uniformly distributed points in $D=[-1,1]^3$ and solve~\eqref{eq:l1_problem} and
\eqref{eq:phi_problem} with \(\phi_\gamma\) from \eqref{eq:func_phi_gamma}
and $\gamma=5$. To evaluate the influence of \(\alpha\), we consider \(\alpha =
10^{-2} \cdot (1/4)^j\) for \(j=0,1,\ldots,3\).
For both functionals, the algorithm is initialized with the previous solution (for larger
\(\alpha\), if available), iterated \fixed{\(T=10\)} times with
up to \fixed{\(N_{\text{trial}} = 50\)} nodes added in each iteration.
In Figure~\ref{fig:3d} we investigate more closely the properties of the solutions for the
smallest \(\alpha\), which again illustrates the improved sparsity of the nonconvex approach.
Moreover, Table~\ref{tab:3d} confirms that this reduction holds independently of
\(\alpha\), that the fidelity is comparable between the convex and nonconvex solution, and
that \(\alpha\) determines the quality of the fit. We note that the error is reduced by a
factor \(\sim 1/2\) when \(\alpha\) is reduced by a factor \(1/4\), which is in accordance with Theorem~\ref{thm:fid_est}. 
}
\begin{figure}[h!]
  \centering
  \begin{subfigure}{.45\linewidth}\centering
    \includegraphics[width=0.99\textwidth]{paper_test_4/l1_forward.tikz}
    \caption{$\alpha=\num{1.6e-4}$, $\gamma = \num{0}$.}
  \end{subfigure}
  \begin{subfigure}{.45\linewidth}\centering
    \includegraphics[width=0.99\textwidth]{paper_test_4/phi_forward.tikz}
    \caption{$\alpha=\num{1.6e-4}$, $\gamma = \num{5}$.}
  \end{subfigure}
	\caption{
          Comparison of solutions of~\eqref{eq:l1_problem} and
          \eqref{eq:phi_problem} with \(\phi_\gamma\) from~\eqref{eq:func_phi_gamma} for
          \(f\) from~\eqref{eq:f_3d}:
          Location of inner weights (dot color indicates sign, dot size
          outer weight magnitude) under stereographic projection of~\(\mathbb{S}^3\) from
          the pole \((a,b)=(0,0,-1)\).
        }\label{fig:3d}
\end{figure}
\begin{table}
  \centering
\begin{tabular}{rrlrl}
\toprule
 & \multicolumn{2}{c}{($\gamma=0$)} & \multicolumn{2}{c}{($\gamma=5$)} \\
$\alpha$ & $N$     & $\norm{\NN - f}$ & $N$     & $\norm{\NN - f}$ \\
\midrule
\num{1e-2}       & {31}  & \num{0.84} & {9}   & \num{0.60} \\
\num{2.5e-3}     & {60}  & \num{0.51} & {15}  & \num{0.35} \\
\num{6.3e-4}     & {83}  & \num{0.31} & {22}  & \num{0.19} \\
\num{1.6e-4}     & {127} & \num{0.15} & {38}  & \num{0.098} \\
\bottomrule
\end{tabular}
\vspace{.2em}
\caption{Optimal $N$ and fidelity for the solutions of~\eqref{eq:l1_problem} and~\eqref{eq:phi_problem}
  with~\eqref{eq:func_phi_gamma} for $\gamma=5$ and different $\alpha$ using $f$
  from~\eqref{eq:f_3d}.}
\label{tab:3d}
\end{table}

\section{Conclusions}
In this work, we introduce a nonconvex regularization method for finding sparse neural
networks. Even though it is challenging to solve the problem globally, we provide
theoretical confirmation of  local minimizers satisfying the desirable sparsity and
approximation properties.   We numerically solve the problem using an adaptive
method that gradually adds and removes nodes from the network. 

We assume that the target function outputs scalar values. This
restriction can be easily lifted by considering the outer weights $c_n$ to be
vectors, and by replacing the penalty term $\phi(|c_n|)$ with $\phi(\|c_n\|)$ in the
problem formulation. Then, the generalized problem
with measures will be modified by switching from signed measures to vector-valued
measures. We can expect all the results of this paper to hold true with small
modifications. 

Throughout the paper, we largely limited ourselves to the ReLU activation function,
although  the optimization problem is not specific to this choice and many activation
functions that fulfill the Lipschitz continuity requirement  could be employed
instead. However, the ReLU allowed us to restrict the inner weights from $\R^{d+1}$ to $\Sd$.
For other activation functions that are not positively homogeneous (in particular smoother
ones), the restriction
from the whole space to a bounded set can impose restrictions on the
approximation properties of the resulting architecture, and it may be necessary to
consider an unbounded dictionary in \(\R^{d+1}\).

A more thorough investigation of the numerical optimization algorithm,
algorithmic choices, and implementation details may warrant more attention.
Additionally, we only consider nonconvex regularization and adaptive training for  shallow
neural networks.
For the case of deep neural networks, a direct reformulation in terms of appropriate
integral representations remains a challenging problem.




\bibliography{refs}
\bibliographystyle{siam}


\appendix

\section{Properties of \(\Phi\) and the problem~\eqref{eq:Phi_problem}}\label{sec:phi_reg}

\fixed{
First, we discuss some simple consequences of the requirements~\ref{cond:phi_1} and~\ref{cond:phi_2}.
\begin{proposition}
Let \(\phi\colon \R_+ \to \R_+\) be twice differentiable almost everywhere with
\(\phi(0) = 0\), \(\phi^\prime(0) = 1\) and
\[
  -\gamma \leq \phi^{\prime\prime}(z) \leq 0 \quad\text{for almost all } z \geq 0. 
\]
Then \(\phi\) fulfills~\ref{cond:phi_1} if and only if \(\phi(z) \to \infty\) for \(z\to \infty\).
Moreover, if for some \(0 < \widehat{\gamma} \leq \gamma\) and \(\widehat{z}>0\),
\[
  \phi^{\prime\prime}(z) \leq -\widehat{\gamma} \quad\text{for almost all } z
  \in (0, \widehat{z}),
\]
then \(\phi\) fulfills~\ref{cond:phi_2}.
Moreover, \ref{cond:phi_1} resp.~\ref{cond:phi_2} imply that
\begin{align*}
  z - (\gamma/2) z^2 \leq \phi(z) &\leq z &&\text{for } z \geq 0, \\
  \phi(z) &\leq z - (\widehat{\gamma}/2) z^2 &&\text{for } z \in [0, \widehat{z}].
\end{align*}
\end{proposition}
\begin{proof}
By integrating the first condition from \(0\leq z_1 \leq z_2\) we obtain the desired condition
\[
  -\gamma(z_2-z_1) \leq \phi^\prime(z_2) - \phi^\prime(z_1) \leq 0
  \quad\text{for all } 0 \leq z_1 \leq z_2,
\]
which implies that \(\phi\) is concave and \(\gamma\)-convex, which are all conditions
from~\ref{cond:phi_1} except for the unboundedness at infinity.
Plugging in \(z_1 = 0\), using \(\phi'(0) = 1\) and \(\phi(0) = 0\) and integrating over
\(z_2\) from \(0\) to \(z\) yields the second inequality.
Similarly, integrating the second inequality once yields the condition in~\ref{cond:phi_2} and
integrating twice yields the second inequality.
\end{proof}
\begin{proposition}
Condition~\ref{cond:phi_1} implies that \(\phi\) is subadditive.
\end{proposition}
\begin{proof}
For any \(z_1\geq 0\), \(z_2\geq 0\) there holds
\[
  \phi(z_2+z_2)
  = \int_0^{z_1} \phi^\prime(z) \de z + \int_{0}^{z_2} \phi^\prime(z_1+z) \de z
  \leq \int_0^{z_1} \phi'(z) \de z + \int_0^{z_2} \phi'(z) \de z
  =\phi(z_1) + \phi(z_2),
\]
using that \(\phi^\prime(z_1+z) \leq \phi^\prime(z)\).
\end{proof}
With this, it is easy to verify that the functions \(\phi_{\log,\gamma}(z)\) and
\(\operatorname{MCP}_\gamma\) mentioned in section~\ref{sec:intro} fulfill the conditions. In the former
case, we have
\[
\phi_{\log,\gamma}^{\prime\prime}(z) = -\frac{\gamma}{(1+\gamma z)^2},
\]
which fulfills the conditions on the second derivative with, e.g., \(\widehat{\gamma} =
\gamma/4\) and \(\widehat{z} = 1/\gamma\), and in the latter case we have
\[
  \operatorname{MCP}_\gamma^{\prime\prime}(z)
  = \begin{cases}
    -\gamma &\text{for } z \leq 1/\gamma \\
    0 &\text{else,}
  \end{cases}
\]
where we can take \(\widehat{\gamma} =\gamma\) and \(\widehat{z} = 1/\gamma\).
We also note that for any \(\phi_\gamma\) that fulfills the conditions on the second
derivative with \(\widehat{\gamma} =
\widehat{\gamma}_{1} \gamma\) and \(\widehat{z} = \widehat{z}_1/\gamma\), the convex combination
\[
  \phi(z) = \tau \phi_{\gamma/\tau}(z) + (1-\tau) z,
  \quad 0<\tau<1,
\]
fulfills the conditions~\ref{cond:phi_1} and~\ref{cond:phi_2} with the same \(\gamma\), \(\widehat{\gamma}\) and
\(\widehat{z} = \tau \widehat{z}_1/\gamma\) using that \(\phi(z) \geq
(1-\tau)z \to \infty\) for \(\tau \to \infty\).
}

\begin{lemma}\label{lem:phi_ineq}
\fixed{Assume~\ref{cond:phi_1} is fulfilled.} For any \(\mu \in M(\Omega)\) it holds
\[
\phi(\norm{\mu}_{M(\Omega)}) \leq \Phi(\mu) \leq \norm{\mu}_{M(\Omega)}.
\]
\end{lemma}
\begin{proof}
The second inequality follows directly from \(\phi(|z|) \leq \abs{z}\).
Concerning the first, we let \(\atom\mu = \{\,\omega_n\,\}_n \) be the atoms of $\mu$
and estimate
\begin{align*}
\Phi(\mu) &\geq  \phi\left(\abs{\mu}\left(\Omega\setminus \atom\mu\right)\right)
+ \sum_{n} \phi\left(\abs{\mu}(\{\omega_n\})\right) \\
&\geq \phi\left(\abs{\mu}\left(\Omega\setminus\atom\mu\right) + \sum_{n} \abs{\mu}(\{\omega_n\}) \right)
= \phi(\norm{\mu}_{M(\Omega)}),
\end{align*}
where we used first the subadditivity of \(\phi\) and second the \(\sigma\)-additivity of \(\fixed{\abs{\mu}}\).
\end{proof}

For a sequence of measures $\mu^{(k)}\in M(\Omega)$, $k=1,2,\dots$,  and a measure $\mu\in
M(\Omega)$, denote $\mu^{(k)}\rightharpoonup^* \mu$ if $\mu^{(\fixed{k})} \in M(\Omega)$ converges
to $\mu$ in weak-$*$ sense as functionals on $C(\Omega)$, i.e.\ for any
$\fixed{\varphi}\in C(\Omega)$,
\[
\lim_{\fixed{n}\to\infty}\pair{\varphi, \mu^{(k)}}=\pair{\varphi, \mu}.
\]
The next lemma can be derived from Theorem~3.3 in \cite{bouchitte1990new}. 
\begin{lemma}
\fixed{Assume~\ref{cond:phi_1} is fulfilled.}
 \(\Phi\) is weak-\(*\) lower semicontinuous on \(M(\Omega)\): if $\mu^{(k)}\rightharpoonup^* \mu$ then 
 \begin{equation}\label{cond:lower_semi}
     \liminf_{\fixed{n}\to\infty}\Phi(\mu^{(k)})\geq \Phi(\mu).
 \end{equation}
\end{lemma}
\begin{proof}
  \fixed{%
  The functional \(\Phi\) can be written as
  \[
    \Phi(\mu)=\int_{\Omega\setminus \atom\mu} \de\abs{\mu} + \sum_{\omega\in \atom\mu} \phi(\abs{\mu(\{\,\omega\,\})}),
  \]
  which is an instance of the functional discussed in \cite{bouchitte1990new} (where we
  choose \(f(z) = \abs{z}\)). We verify
  the assumptions of \cite[Theorem~3.3]{bouchitte1990new}, which gives the desired result: The
  assumptions $(H_1)$--$(H_3)$ are trivially fulfilled. Continuity and
  subadditivity of \(z \mapsto \phi(\abs{z})\),
  i.e., \(\phi(\abs{z_1+z_2}) \leq \phi(\abs{z_1}+\abs{z_2}) \leq
  \phi(\abs{z_1})+\phi(\abs{z_2})\), imply assumptions $(H_4)$--$(H_5)$, assumption
  $(H_6)$ follows from \(\phi(\abs{z}) \leq \abs{z}\), and assumption
  $(H_7)$ simplifies to \(\phi'(0)\abs{z} = \abs{z}\) for all \(c \in \R\), i.e.\ \(\phi'(0) = 1\).%
  }
\end{proof}
\fixed{With these lemmas we can give the basic existence result from section~\ref{subsec:local_sol}.}
\begin{proof}[Proof of Theorem~\ref{thm:local_attainable}]
We employ the direct method of variational calculus.
Since \(J \geq 0\), we can select a minimizing sequence \(\mu^{(\fixed{k})} \in M(\Omega)\). Now, with
\[
 \fixed{\alpha}\phi(\norm{\mu^{(\fixed{k})}}_{M(\Omega)}) \leq \fixed{\alpha}\Phi(\mu^{(\fixed{k})})
 \leq J(\mu^{(\fixed{k})}) \to \inf_{\mu \in M(\Omega)}J(\mu)
\]
and the fact that \(\phi(z) \to +\infty\) for \(z\to\infty\),
\(\norm{\mu^{(\fixed{k})}}_{M(\Omega)}\) is bounded and we can select a subsequence that converges
to \(\bar{\mu} \in M(\Omega)\) in the weak-\(*\) sense.
By weak-\(*\) lower semicontinuity of \(\Phi\) and continuity of the loss function
\(L(\mu)\), we conclude that the minimum is attained at \(\bar{\mu}\).
\fixed{The bound on the global solution follows with
\(\alpha\phi(\norm{\bar\mu}_{M(\Omega)}) \leq \alpha\Phi(\bar\mu) \leq J(\bar{\mu}) \leq
J(0) = l(0;y)\).
}
\end{proof}

\section{Proof of Theorem \ref{thm:opt_cond}}
\label{sec:opt_cond}

In the following, we let \(\bar{\mu}\) be a local solution of~\eqref{eq:Phi_problem},
i.e.\ a local minimum of \(J\) (given in~\eqref{eq:J-func}).
To derive the conditions for the given local solution, we consider 
\[
J(\bar{\mu} + \tau \, u) - J(\bar{\mu}) = L( \bar{\mu} + \tau \, u ) - L(\bar{\mu}) 
+ \alpha\Phi( \bar{\mu} + \tau \,u ) - \alpha \Phi(\bar{\mu})\geq 0
\]
with arbitrary \(u \in M(\Omega)\) and \(0 < \tau < \epsilon/\norm{u}_{M(\Omega)}\) where
$\epsilon$ is the radius from Definition~\ref{def:local_min_meas}.
Dividing by \(\tau>0\) and letting \(\tau \to 0\), from local optimality, it follows
\[
  \fixed{-} \pair{\nabla L(\bar{\mu}), u} =
  - \pair{\bar{p}, u}  \leq \alpha
  \lim_{\tau \to 0^+} (1/\tau) \left[\Phi(\bar{\mu} + \tau\, u) - \Phi(\bar{\mu})\right],
\]
as long as the limit on the right exists.
We consider different values of \(u\) in the following:

For \(u =  \pm\delta_{\omega}\) for \(\omega \not\in \atom\fixed{\bar\mu}\), i.e.\ \(\fixed{\bar\mu}(\{\omega\}) =  0\), we obtain
\[
\mp\bar{p}(\omega) = -\pair{\bar{p}, u} \leq \alpha \lim_{\tau \to 0^+} (1/\tau) \phi(\tau)=
\alpha \lim_{\tau \to 0} (1/\tau) \left[\phi(\tau) - \phi(0)\right] 
= \alpha \, \phi'(0) = \alpha.
\]
Hence, $|\bar{p}(\omega)|\leq \alpha$ for \(\omega \not\in \atom\fixed{\bar\mu}\). 

Now, take \(u = \pm\delta_{\omega}\) for \(\omega \in \atom\fixed{\bar\mu}\),
i.e.\ \(\bar{\mu}(\{\omega\}) = c \neq  0\). Here, we obtain
\[
\mp\bar{p}(\omega) = \fixed{-}\pair{\bar{p}, u}
\leq \alpha \lim_{\tau \to 0^+} (1/\tau) \left[\phi(\abs{c \pm \tau}) - \phi(\abs{c})\right]
 = \pm\alpha \, \phi'(|c|) \, \sign c.
\]
Hence, it follows
\[
\bar{p}(\omega) = - \alpha \, \phi'(|c|) \, \sign c.
\]
Using the fact that $\phi^\prime(0)=1$, $\phi$ is increasing and $\phi^\prime$ is
decreasing we have that  $\phi^\prime(\omega)\in [0,1]$ for $\omega\geq 0$.
Hence $|\bar{p}(\omega)|\leq \alpha$ for \(\omega \in \atom\fixed{\bar\mu}\) also proving the first
estimate in the lemma: $|\bar{p}(\omega)|\leq \alpha$ for all \(\omega \in \Omega\).

Next, we take \(u = \fixed{\pm}\bar{\mu}_{\text{cont}} = \fixed{\pm}\bar{\mu}\rvert_{\Omega\setminus\atom\fixed{\bar\mu}}\)
to be the continuous part of $\bar\mu$ and deduce
\[
\mp\pair{\bar{p}, \bar{\mu}_{\text{cont}}}
 \leq \alpha \lim (1/\tau)[(1\pm\tau)\norm{\bar{\mu}_{\text{cont}}}_{M(\Omega)} - \norm{\bar{\mu}_{\text{cont}}}_{M(\Omega)}]
 = \pm \alpha \norm{\bar{\mu}_{\text{cont}}}_{M(\Omega)}
\]
and thus \(-\pair{\bar{p}, \bar{\mu}_{\text{cont}}} = \alpha\norm{\bar{\mu}_{\text{cont}}}_{M(\Omega)}\),
which together with \(\abs{\bar{p}(\omega)} \leq \alpha\) for $\omega\in\Omega$
implies that \(\bar{p} = -\alpha\sign{\bar{\mu}_{\text{cont}}}\) for
$\bar{\mu}_{\text{cont}}$ almost all $\omega \in \Omega$.
Combined with the atomic case above we get the second part of the theorem.

\section{Proof of Theorem \ref{thm:local_finiteness}}
\label{sec:local_finiteness}

First, let us show that $\bar\mu$ is atomic. 
Assume otherwise and let
\(\bar{\mu}_{\text{cont}} = \bar{\mu}\rvert_{\Omega\setminus\atom\fixed{\bar\mu}} \neq 0\)
be the continuous part of \(\bar{\mu}\).
Then there exist a point $\hat{\omega} \in \supp\bar{\mu}_{\text{cont}} \setminus \atom\fixed{\bar\mu}$, i.e., such that
$\bar\mu(\{\hat\omega\})=0$ and for any $\delta > 0$,
$\abs{\bar{\mu}}(B_\delta(\hat{\omega})) > 0$ where
$B_\delta(\hat{\omega})$ is the open ball of radius $\delta$ in $\Omega$ around
$\hat\omega$.
Without restriction, let \(\hat{\omega} \in \supp\bar{\mu}_{\text{cont},+}\) where
\(\bar{\mu}_{\text{cont},+}\) is the positive part of \(\bar{\mu}_{\text{cont}}\).
Let \(D_+ \subset \Omega\setminus\atom\fixed{\bar\mu}\) be a set with \(\bar{\mu}_{\text{cont},+} = \bar{\mu}\rvert_{D_+}\)
(given by the Hahn decomposition theorem) and \(D_\delta = D_+ \cap B_\delta(\hat{\omega})\).
Now, we define
\[
\mu_\delta = \bar{\mu} - \bar{\mu}\rvert_{D_\delta}
+ C_\delta\delta_{\hat{\omega}},
\quad\text{where }
C_\delta = \bar{\mu}(D_\delta) > 0 
\]
which replaces $\bar{\mu}$ on \(D_\delta\) for any \(\delta > 0\) by a single Dirac
measure of the same total variation norm.
We note that by construction \(C_\delta\) is positive and converges to zero for \(\delta \to 0\).
We will show that \(\mu_\delta\), for \(\delta\) small enough, improves the function value of $J$ in contradiction to  optimality of \(\bar{\mu}\).

Using the Lipschitz continuity assumption on  \(\sigma\),
i.e.\ \(\abs{\sigma(x;\hat{\omega}) - \sigma(x;\omega)} \leq\Lambda\fixed{(1+\|x\|)\norm{\omega - \hat\omega}}\),
one readily obtains that
\begin{multline*}
   \abs*{[\NN(\mu_\delta - \bar{\mu})](x)}
 = \abs*{[\NN(\bar{\mu}\rvert_{D_\delta} - C_\delta \delta_{\hat{\omega}})](x)}
\leq \abs*{\int_{D_\delta}\sigma(x;\omega)\dd\bar{\mu}(\omega)
- C_\delta\sigma(x;\hat{\omega})}\\
=\abs*{\int_{D_\delta}[\sigma(x;\omega)-\sigma(x;\hat{\omega})]\dd\bar{\mu}(\omega)}
\leq\delta\Lambda\fixed{(1+\|x\|)}\abs{\bar\mu}(D_\delta)
  =\delta\Lambda\fixed{(1+\|x\|)}\bar\mu(D_\delta)
  =\delta\Lambda\fixed{(1+\|x\|)} C_\delta
\end{multline*}
for any \(x\in \R^d\).
Thus, it also follows
\[
\norm{\NN(\mu_\delta -\bar{\mu})}_{\Hi}
\leq  \delta \Lambda \fixed{C_\nu} C_\delta
= \delta \Lambda_1 \, C_\delta.
\]
Here, we define \fixed{$C_\nu^2 = \int_D \left(1+\norm{x}\right)^2 \de\nu(x)$}
and  \(\Lambda_1 = \Lambda \fixed{C_\nu}\).
Consequently, by the quadratic form of \(L\) and, using the fact that $\bar p= \NN^*
(\NN\bar\mu - y) $, we have
\begin{align*}
L(\mu_\delta)
&= L(\bar{\mu})
+ (\NN \bar{\mu} - y, \NN(\mu_\delta - \bar{\mu}) )_{\Hi} + \frac{1}{2}\norm{\NN(\mu_\delta - \bar{\mu})}^2_{\Hi} \\
&\leq L(\bar{\mu}) + \pair{\bar{p}, \mu_\delta - \bar{\mu}}
+ \frac{1}{2} \delta^2 \Lambda^2_1 C_\delta^2.
\end{align*}
By the optimality condition \fixed{and continuity of \(\bar{p}\)}, we have \(\bar{p}(\omega) = -\alpha\) for all
\(\omega \in D_\delta \cap \supp\bar{\mu} \subset \supp\bar\mu_{\text{cont},+}\)
(note that also $\hat\omega \in D_\delta\cap \fixed{\supp\bar{\mu} \subset \supp\bar\mu_{\text{cont},+}}$) and therefore the term
\(\pair{\bar{p}, \mu_\delta - \bar{\mu}} = \pair{\bar{p}, -\bar{\mu}\rvert_{D_\delta}
+ C_\delta\delta_{\hat{\omega}}}\) vanishes.
Hence,
\[
  L(\mu_\delta) \leq L(\bar{\mu}) 
  + \frac{1}{2}\delta^2\Lambda_1^2 C_\delta^2.
\]
Concerning the cost term, we estimate 
\begin{align*}
\Phi(\mu_\delta)&
=  \Phi(\bar{\mu}) - \bar{\mu}(D_\delta) + \phi(C_\delta )
= \Phi(\bar{\mu}) + \phi(C_\delta ) - C_\delta\\
&= \Phi(\bar{\mu}) + \int_0^{C_\delta} (\phi'(\xi) - \phi'(0)) \de \xi
\leq
\Phi(\bar{\mu})
\fixed{-} \widehat{\gamma} \int_0^{C_\delta} \fixed{\xi} \de \xi 
= \Phi(\bar{\mu}) - \frac{\widehat{\gamma}}{2} C_\delta^2,  
\end{align*}
for \(\delta\) small enough,
using \(\phi'(0) = 1\), the definition of \(\Phi\) and~\ref{cond:phi_2}.
Combining both estimates, we obtain
\[
J(\mu_\delta) = L(\mu_\delta) + \alpha\Phi(\mu_\delta) \leq J(\bar{\mu}) 
+ \frac{1}{2}\left(\delta^2 \Lambda^2_1 - \alpha \widehat{\gamma}\right) C_\delta^2.
\]
Therefore
\[
J(\mu_\delta) 
\leq J(\bar{\mu})
 - \frac{1}{2}\left(\alpha\widehat{\gamma} - \delta^2\Lambda^2_1\right)C_\delta^2 < J(\bar{\mu}),
\]
for \(\delta < \sqrt{\alpha\widehat{\gamma}}/\Lambda_1\), contradicting the optimality of \(\bar{\mu}\).

Now let us show that the number of atoms in $\bar\mu$ is finite. Assume otherwise, then
there exists a subsequence of distinct atoms $\bar{\omega}_n$ converging to some $\hat \omega$
due to compactness of $\Omega$. Without restriction, assume that $c_n =
\bar\mu(\{\bar{\omega}_n\}) > 0$ for all \(n\).
By optimality of \(\bar{\mu}\) and continuity of \(\bar{p}\), from Theorem \ref{thm:opt_cond}, it holds
\[
\alpha \phi'(c_n) = -\bar{p}(\bar{\omega}_n) \to -\bar{p}(\hat{\omega})
\quad\text{for } n \to \infty.
\]
Since \(\phi'(c_n) \to 1\) due to \(c_n \to 0\), it follows that \(\bar{p}(\hat{\omega}) =
-\alpha\). Hence, from Theorem~\ref{thm:opt_cond}, $\hat{\omega}$ cannot be an atom of
$\bar\mu$. Define now
\[
\mu_N = \bar{\mu} - \sum_{n=N}^\infty c_n \delta_{\bar{\omega}_n} + C_N\delta_{\hat{\omega}}
\quad\text{where } C_N = \sum_{n=N}^\infty c_n > 0,
\]
replacing an infinite number of atoms by a single one.
Set \(\delta_N = \max_{n\geq N} \abs{\bar{\omega}_n - \hat{\omega}}\). 
As before, we obtain
\begin{align*}
L(\mu_N)
&\leq L(\bar{\mu}) + \pair{\bar{p}, \mu_N - \bar{\mu}} + \frac{1}{2}\delta_N^2 C_N^2 \Lambda_1^2.
\end{align*}
Here, the second term is given as
\[
\pair{\bar{p}, \mu_N - \bar{\mu}}
=   C_N \bar{p}(\hat{\omega}) - \sum_{n=N}^\infty c_n \bar{p}(\bar{\omega}_n)
= - \alpha C_N + \alpha \sum_{n=N}^\infty c_n \phi'(c_n),
\]
using the optimality conditions.
Concerning \(\Phi\), there holds 
\[
\Phi(\mu_N) = \Phi(\bar{\mu}) - \sum_{n=N}^\infty\phi(c_n) + \phi(C_N).
\]
Combining these estimates, we obtain
\[
J(\mu_N)
= L(\mu_N) + \alpha\Phi(\mu_N)
\leq J(\bar{\mu}) - \alpha \sum_{n=N}^\infty \left[ \phi(c_n) - \phi'(c_n) c_n \right] - \alpha [C_N - \phi(C_N)]
+ \frac{1}{2}\delta_N^2 C_N^2\Lambda_1^2.
\]
Now, we use concavity of \(\phi\) for \(\phi(c_n) - \phi'(c_n) c_n \geq \phi(c_n-c_n) =
\phi(0) = 0\) and uniform concavity of \(\phi\) on $[0,\widehat{z}]$, using~\ref{cond:phi_2},
for \(\phi(C_N) = \phi(0) + \phi'(0)C_N + \int_0^{C_N} [\phi'(\xi) - \phi'(0)] \de \xi
\leq C_N - (\widehat{\gamma}/2) C_N^2\) when $N$ is large enough, to obtain
\[
J(\mu_N)
\leq J(\bar{\mu}) - \frac{1}{2}\left(\alpha\widehat{\gamma} - \delta_N^2\Lambda_1^2\right) C_N^2.
\]
Similar to the previous case, we choose now \(N\) such that \(\delta_N <
\sqrt{\alpha\widehat{\gamma}}/\Lambda_1\), which again results in a
contradiction to the optimality of \(\bar{\mu}\).

This, together with the optimality conditions obtained in Theorem \ref{thm:opt_cond}, concludes the proof.

\section{Proof of Theorem \ref{thm:representer_phi}}\label{sec:representer}

For the total variation norm (i.e. \(\phi(z) = z\)), the Carath{\'e}odory lemma implies an upper bound on the number of atoms for some optimal solution \(\bar{\mu}\) in some cases.
In particular, we consider the special case of finitely supported \(\nu\), which is given by a sum of \(K\) Dirac delta measures. In this case, the space \(\Hi\) is finite dimensional, i.e.\ \(\dim \Hi = K\).
To prove the Theorem \ref{thm:representer_phi} we need to show that any local solution of~\eqref{eq:Phi_problem} is atomic and its support consists of at most \(K\) points.

By the previous result we know that any locally optimal solution is representable as
\[
\bar{\mu} = \sum_{n=1}^N \bar{c}_n \delta_{\bar{\omega}_n},
\quad\abs{\bar{c}_n} > 0,\ \bar{\omega}_n \in \Omega,\, N\in \N .
\]
Clearly,
\[
\NN\bar{\mu} =
\sum_{n = 1}^{N} \bar{c}_n \NN\left( \delta_{\bar{\omega}_n} \right)
= \sum_{n=1}^{N} \bar{c}_n \sigma(\cdot;\bar{\omega}_n).
\]
 Assume that \(N > K\). Then there exists a nontrivial vector \(\lambda \in \R^{N}\) such that
\[
\sum_{n = 1}^{N} \lambda_n c_n \sigma(x_k;\bar{\omega}_n) = 0 \quad \text{for all } k=1,\ldots,K
\]
 or, equivalently,  \(\sum_{n = 1}^{N} \lambda_n c_n \sigma(\cdot;\bar{\omega}_n) = 0 \) in \(\Hi\).
For any \(\tau \in \R\) we define 
\[
\bar{\mu}_\tau = \sum_{n = 1}^{N} (1 + \tau \lambda_n) \bar{c}_n \delta_{\bar{\omega}_n} .
\]
Note that \(\NN\bar{\mu}_\tau = \NN\bar{\mu} + \tau \sum_{n=1}^{\fixed{N}} \lambda_n c_n \sigma(\cdot;\bar{\omega}_n) = \NN\bar{\mu}\) in \(\Hi\) for any \(\tau\). 
Now, we assume also that $\tau$ is small enough such that
\(1+\lambda_n \tau \geq 0\) for all \(n\) and
turn our attention to the objective functional of~\eqref{eq:Phi_problem}. Taking into account the previous argument, we have, for any \(\tau\neq 0\), that
\begin{align}
J(\bar{\mu}_\tau) - J(\bar{\mu}) = \fixed{\alpha}\Phi(\bar{\mu}_\tau) - \fixed{\alpha}\Phi(\bar{\mu})
&= \alpha\sum_{n=1}^{N} \left[ \phi\left((1+\tau \lambda_n) \abs{\bar{c}_n}\right) - \phi(\abs{\bar{c}_n}) \right] \nonumber\\
&< \tau \alpha \sum_{n=1}^{N} 
\phi'\left(\abs{\bar{c}_n}\right) \lambda_n  \abs{\bar{c}_n} 
= \tau \alpha\, \Phi'(\bar{\mu}; \delta\mu),
\nonumber
\end{align}
where \fixed{\(\Phi'(\bar{\mu}; \delta\mu)\) is the directional derivative of \(\Phi\) at
\(\bar\mu\) in direction} \(\delta \mu = \bar{\mu}_1 - \bar{\mu}\),
taking into account that \(\fixed{L}(\bar{\mu}_\tau) = \fixed{L}(\bar{\mu})\), the restrictions on \(\tau\) and the strict concavity of \(\phi\).
Depending on the sign of \(\Phi'(\bar{\mu}; \delta\mu)\), we choose \(\tau > 0\) or \(\tau < 0\)
sufficiently small such that
\[
J(\bar{\mu}_\tau) - J(\bar{\mu}) < 0,
\]
contradicting the local optimality of \(\bar{\mu}\).

\section{The optimization algorithm}\label{sec:sup_gen_conj_grad}


In the following, we
give more details on the concrete implementation of the steps summarized
in Algorithm~\ref{alg:practical}.

\paragraph{Initialization:}
Let $\omega^{(t)} = [\omega^{(t)}_1,\ldots,\omega^{(t)}_{N(t)}]$,
$c^{(t)}=[c_1^{(t)},\ldots,c^{(t)}_{N(t)}]$ be the lists of network inner and outer weights
in the $t$-th iteration and $N(t)$ be corresponding the number of neurons.
In addition, denote by $\NN_{\omega^{(t)},c^{(t)}}(x)$ the corresponding network. 
Network initialization is arbitrary:
one can start from any network, including the empty network, then add
and extract nodes to derive an optimal network.

\paragraph{Phase 1}
To determine new points to insert, in the greedy insertion step, we compute the
nodes \(\omega\in\Omega\) for which the correlation of $\sigma(\omega,x)$ with the residual 
\(g_t(x) = \NN_{\omega^{(t)},c^{(t)}}(x)-y(x)\) is largest. Thus, we maximize the absolute
value of
\[
  p_t(\omega)
  =
  \int_{D}\sigma(\omega, x)\,g_t(x)\de \nu (x)
  =
  \frac{1}{K}\sum_{k=1}^K \sigma(\omega, x_k) g^{(t)}_k
\]
where we assume that \(K\) is finite and \(g^{(t)}_k = g_t(x_k)\).
Note that $p_t(\omega)$ is exactly the dual variable as defined in
Subsection~\ref{subsec:local_sol}.
Finding a global maximum in a high dimensional space is a challenging problem which
requires expensive computations, and its reliable determination up to a guaranteed
tolerance for the specific problem here is subject of ongoing research; cf.\ \cite{bach2017breaking}.
As an ersatz, we use the following heuristic which is
commonly employed in practice: we test all local maxima of
\begin{equation}
  \label{eq:sup_max_dual}
  \Omega \ni \omega \mapsto \abs{p_t(\omega)},
\end{equation}
which are found by a gradient maximization, initialized at \(N_{\text{trial}}\) random
points on $\Omega$.  This corresponds to solving \(N_{\text{trial}}\)  simple
unconstrained optimization problems (in parallel); cf., e.g.,~\cite{boyd2017alternating}. Of
these points, we insert all that violate the constraint \(|p_t(\omega)| \leq \alpha\) (after
removing possible duplicates).
Here, we rely on the random initialization of the \(N_{\text{trial}}\) problems in order
to have a chance to identify the global maximum with some probability. Moreover, more than
one identified local maximum can be added to the network in each iteration
to identify a potentially wide network faster. According to Theorem~\ref{thm:nec_suff_local_sol},
at any node $\omega\in\Omega$ where \(|p_t(\omega)| < \alpha\), it is not possible to
decrease the objective by inserting the corresponding node with small non-zero
weight. Conversely, the nodes $\omega\in\Omega$ where \(|p_t(\omega)| > \alpha\),
represent locations where local decrease can still be achieved. All the corresponding
outer weights at the trial nodes are initialized to zero, and are going to be optimized in
the next phase of the algorithm.

Before we address the next phase, we point out that in the case when we consider $\Omega$
to be the sphere and $\sigma$ be the ReLU activation function, we opt to
parametrize \(\omega\) by its stereographic projection
\(\omega = (a,b) = (2 \fixed{z}, 1 - \norm{z}^2) / (1 + \norm{z}^2)\), where \(z \in \R^d\). This
is done to avoid dealing with the algebraic constraint \(\norm{\omega} = 1\). Here, we
use the southern pole as the projection point, which corresponds to $(a,b)= (0, -1)$.
The corresponding neuron represents the zero function and removing it from $\Omega$ does not
affect the approximation capability of the network.

\paragraph{Phase 2}
Let $N(t+1/2)\leq N(t)+N_{\text{trial}}$ denote the number of nodes in the resulting
network. Next, we compute an approximate local solution to
the following problem
\begin{equation}
  \label{eq:sup_local_training}
  (\Omega\times\R)^{N(t+1/2)} \ni (\fixed{\omega},c) \mapsto
  l\left(\NN_{\omega, c};y\right) + \alpha\sum_{n=1}^{N(t+1/2)}\phi(\abs{c_n})
\end{equation}
in terms of all weights, using the old values as initialization for the weights from the
previous iteration. The resulting nonsmooth optimization problem  can be \fixed{approximately} solved by
standard  training methods based on (proximal) gradient descent. In particular, we can
eliminate the constraint for the inner weights by stereographic projection and use
gradient descent, and treat the outer weights with proximal
gradient descent (see, e.g., \cite{parikh2014proximal}). We note that the proximal map for the cost
term \(\phi\) can still be determined uniquely for small stepsize due to
\(\gamma\)-convexity. For instance, for the function
\(\phi_{\gamma}(z) = \log(1 + \gamma z)/\gamma\) from the
introduction~\eqref{form:phi_log}, it is given as
\begin{align*}
\Prox_{\lambda \phi_\gamma}(q)
 &= \argmin_{c \in \fixed{\R}} \frac{1}{2} \fixed{\abs{c - q}^2} + \lambda\phi_\gamma(c) \\
 &=
\frac{\sign q}{2\gamma}
\begin{cases}
  \left(\gamma\abs{q} - 1\right) + \sqrt{\left(\gamma\abs{q} - 1\right)^2
    + 4\gamma\left(\abs{q} - \lambda\right)},
  &\quad\text{for } \abs{q} > \lambda, \\
  0 &\quad\text{else.}
\end{cases}
\end{align*}
In the case that \(\lambda < 1/\gamma\), this proximal mapping is uniquely determined.
We note that this has the potential to set a number of outer weights to zero
because of the proximal descent step. 

We do not specify the details here (in particular, the optimal choices of the stopping criteria for
the different nonlinear and nonsmooth optimization routines that we employ), and leave a
detailed analysis to future work.

\paragraph{Phase 3}
Finally, we note that the objective~\eqref{eq:local_training} contains a sparsity
promoting term for the outer weights, and we can expect several of the outer weights to be
zero after having solved the problem up to a specified tolerance. Thus, we can eliminate
these outer weights (together with their corresponding inner weights) from the network
without changing the underlying function. This results in a new network of final width \(N(t+1)\).

\subsection{Second-order methods for outer weights}\label{sec:sup_second_order}
 The aforementioned gradient-based methods in phase two of 
 Algorithm~\ref{alg:practical} are efficient in that they only require derivatives of
 the objective function in terms of inner and outer weights, which are readily available in
 modern computational toolboxes via automatic differentiation. However, these methods
 suffer from slow convergence once we are close to the minimum, and can be slow to
 eliminate redundant nodes \(\omega\) (as already observed in the context of a convex sparse
 problem with measures; cf.~\cite{2019arXiv190409218P}).
 In order to provide
 accurate results (that are not influenced by the choice of the solver), in our numerical
 experiments we employ a second order semi-smooth Newton method for the outer weights.
 In an additional step after step 6.\ of Algorithm~\ref{alg:practical} this solves the problem in
 terms of outer weights up to machine precision for the fixed inner weights
 $\omega^{(t+1/2)}$, which in particular serves to reliably eliminate redundant nodes.

 \begin{equation}\label{eq:node_extract_func}
  l\left(\NN_{ \omega^{(t+1/2)},c};y\right) + \alpha\sum_{n=1}^{N(t+1/2)}\phi(\abs{c_n})
   = F(c) + \fixed{\alpha}\norm{c}_1
 \end{equation} 
 where 
 \[
   F(c) =
   l\left(\NN_{ \omega^{(t+1/2)},c};y\right)
   + \alpha\sum_{n=1}^{N(t+1/2)} \left[\phi(\abs{c_n}) - \abs{c_n}\right]
 \]
 is a continuously differentiable function of $c=[c_1,\dots,c_{N(t+1/2)}]$ \fixed{with
   Lipschitz continuous derivative} due
 to conditions \ref{cond:phi_1} we imposed on $\phi$.
\fixed{For twice continuously differentiable \(\phi\), such as the \(\log\)-penalty
  function, the same property holds for \(F\).}
 Note that if $c$ is a local minimizer of \eqref{eq:node_extract_func} then
 \begin{equation}\label{cond:incusion_grad}
  -\nabla F(c) \in \alpha \partial \|c\|_1
 \end{equation}
 where $\partial \|c\|_1$ is the subdifferential of the $\ell_1$ norm. 
 For $\lambda >0$ and $ q=[q_1,\dots,q_{N(t+1/2)}]$, let $\Prox_\lambda(q)$ be the proximal
 operator of the $\ell_1$ norm (also known as the soft-thresholding operator)
 which modifies each entry of $q$ according to the formula
 \[
   \Prox_\lambda(q)_n = \sign q_n \max\{\fixed{\abs{q_n}-\lambda},0\}.
 \]
 Using a reformulation of the optimality condition in terms of Robinson's normal map, $c$
 satisfies~\eqref{cond:incusion_grad} if and only if $c=\Prox_\lambda(q)$ for some $q$ and
 \begin{equation}\label{eq:prox_condition}
 \nabla F\left(\Prox_\lambda(q)\right) + \frac{\alpha}{\lambda}\left(q - \Prox_\lambda(q)\right) = 0;
 \end{equation}
 see, e.g., \cite[Prop.~3.5]{pieper2015finite}.
 The first advantage of this  reformulation is that the inclusion condition
 \eqref{cond:incusion_grad} is replaced with an equation. This nonsmooth equation can then
 be solved using a
 semi-smooth Newton method (see, e.g., \cite{Ulbrich:2011}), which exhibits locally
 superlinear convergence. In particular, once the optimal sparsity pattern is identified, it
 converges at the quadratic rate of the classical Newton method. The second important advantage
 is that at each iteration, the soft-thresholding operator outputs a sparse $c$: we use this
 feature to drop the zero entries and reduce the number of nodes in the network.

\section{Equivalences of outer- and all-weights regularization}\label{sec:equiv_cond}

Here, we consider networks with the ReLU activation function \(\sigma(\omega,x) =
\max\{\,a\cdot x + b, 0\,\}\), \fixed{\(\omega = (a,b) \in \R^{d+1}\)}, and prove the equivalence of certain cost terms.

\begin{proposition} Let
 \(p \geq 1\) and \(q \geq 1\) and \(r(\omega)\) be a \(1\)-homogeneous functional:
\[
r(\tau\omega) = \abs{\tau}\, r(\omega)
\quad\text{for any } \omega = (a,b) \in \R^{d+1}.
\]
Then the problems 
\begin{equation}\label{eq:r_unconst}
 \min_{N\in \N,\; (a_n,b_n,c_n) \in \R^d\times\R\times\R}\;
l(\NN_{\omega,c};y) + \alpha \sum_{n=1}^N\left[ \frac{1}{p}\abs{c_n}^{p} + \frac{1}{q}r(\omega_n)^q\right],
\end{equation}
and 
\[
\min_{\fixed{N\in \N},\; N\in \N,\; \{c_n\} \in \R^N,\; \{\omega_n=(a_n,b_n):\,r(\omega_n) \leq \fixed{1} \}} l(\NN_{\omega,c};y)
+ \frac{\fixed{2}\alpha}{s} \sum_{n=1}^N \abs{c_n}^{s/2},
\]
where \(s = 2pq/(q+p) = 2/(1/p + 1/q)\) is the harmonic mean of \(p\) and \(q\), are
equivalent; \fixed{i.e.\ their minimum values are the same and the solutions of one
  problem solve the other up to a re-normalization}.
\end{proposition}
\begin{proof}
Note that due to positive homogeneity of the ReLU activation function \(\sigma\), we have
\[
\NN_{\omega,c} = \NN_{\omega_\tau, c_\tau}
\quad\text{where } \omega_\tau = \tau\, \omega, c_\tau = c / \tau .
\]
It is easy to see that the problem \eqref{eq:r_unconst} is equivalent to
\[
\min_{\fixed{N\in \N},\;\{c_n\} \in \R^N,\; \{\omega_n=(a_n,b_n):\,r(\omega_n) \leq 1 \}, \,\tau_n\geq 0}
  l(\NN_{\omega,c};y) + \alpha \sum_{n=1}^N\left[ \frac{1}{p}\abs{c_n/\tau_n}^{p} + \frac{1}{q}\tau_n^q\right].
\]
Moreover, since the first term does not depend on \(\tau\), we can compute \(\tau_n\) as the minimum of 
\[
\tau \mapsto (1/p) \abs{c_n}^p \tau^{-p} + (1/q) \tau^q.
\]
Differentiating with respect to \(\tau\), we obtain
\[
0 = -\abs{c_n}^p\tau_n^{-p-1} + \tau_n^{q-1}
\leftrightarrow
\tau_n^{q+p} = \abs{c_n}^p
\leftrightarrow
\tau_n = \abs{c_n}^{p/(q+p)}.
\]
Inserting the analytical solution for \(\tau_n\) into the cost term above, we obtain
\[
(1/p)\abs{c_n}^{p-p^2/(q+p)} + (1/q)\abs{c_n}^{qp/(q+p)}
= (1/p + 1/q) \abs{c_n}^{qp/(q+p)},
\]
which completes the proof.
\end{proof}

As a corollary, by taking $r(\omega)=\|\omega\|_p$ for $\omega\in \R^{d+1}$, we get that solving the problem
\[
    \min_{\fixed{N\in \N},\;\{c_n\} \in \R^N,\; \{\omega_n=(a_n,b_n)\} \in (\Sd_p)^N}\;
    l\left(\NN_{\omega,c};y\right)
    + \fixed{\frac{2\alpha}{p}}\sum_{n=1}^N|c_n|^{{p}/{2}},
\]where  $\Sd_p = \{\,(a,b)\in\R^{d+1} \;|\; \|a\|_p^p + |b|^p = 1\,\}$ is the unit $p$-sphere in $\R^{d+1}$,  
is equivalent to solving the problem 
\[
  \min_{\fixed{N\in \N},\; (a_n,b_n,c_n) \in \R^d\times\R\times\R}\; l\left(\NN_{\omega,c};y\right)
    + \frac{\alpha}{p}\sum_{n=1}^N\left[\|a_n\|^p_p+|b_n|^p+|c_n|^p\right].
\]

%


\end{document}